%% file: nonlinear_SA.tex
\documentclass[letterpage, 11pt, notitlepage]{article}

\usepackage[margin=1.0in]{geometry}

\usepackage{amsmath, amsfonts, amssymb, graphicx,epstopdf,color,soul,mathtools,enumitem,amsthm} 
\usepackage{mathrsfs, mathtools,bbm, dsfont}
\usepackage{appendix}

\usepackage{ifthen}
\usepackage{multirow, multicol}
\usepackage{float}
\usepackage{subeqnarray, array}
\usepackage{enumitem}
\usepackage{xcolor}
\usepackage[font=small]{caption}
\usepackage{tikz}
\usepackage{times}

\usepackage[bookmarks=true,pageanchor,colorlinks,linkcolor=red,anchorcolor=blue, citecolor=blue,urlcolor=blue,hyperfootnotes=false]{hyperref}

\input{header}
\title{Finite-Time Convergence Rates of Nonlinear Two-Time-Scale Stochastic Approximation under Markovian Noise}
\author{
Thinh T. Doan\thanks{Thinh T. Doan is with the Bradley Department of Electrical and Computer Engineering, Virginia Tech, USA. Email: {\tt\small thinhdoan@vt.edu}}
}
\date{}

\begin{document}

%%%%%%%%%%%%%

\maketitle

\begin{abstract}
We study the so-called two-time-scale stochastic approximation, a simulation-based approach for finding the roots of two coupled nonlinear operators.  Our focus is to characterize its finite-time performance in a Markov setting, which often arises in stochastic control and reinforcement learning problems. In particular, we consider the scenario where the data in the method are generated by Markov processes, therefore, they are dependent. Such dependent data result to biased observations of the underlying operators. Under some fairly standard assumptions on the operators and the Markov processes, we provide a formula that characterizes the convergence rate of the mean square errors generated by the method to zero. Our result shows that the method achieves a convergence in expectation at a rate $\Ocal(1/k^{2/3})$, where $k$ is the number of iterations. Our analysis is mainly motivated by the classic singular perturbation theory for studying the asymptotic convergence of two-time-scale systems, that is, we consider a Lyapunov function that carefully characterizes the coupling between the two iterates. In addition, we utilize the geometric mixing time of the underlying Markov process to handle the bias and dependence in the data. Our theoretical result complements for the existing literature, where the rate of nonlinear two-time-scale stochastic approximation under Markovian noise is unknown. 
\end{abstract}

\input{intro}

\input{results}

\newpage

\bibliographystyle{ieeetr}
\bibliography{refs}

\input{appendix}

% \newpage
% \appendix
% \onecolumn
% \vbox{%
% \hsize\textwidth
% \linewidth\hsize
% \hrule height 4pt
% \vskip 0.25in
% \centering
% {\Large\bf{Nonlinear Two-Time-Scale Stochastic Approximation:\\ Convergence and Finite-Time Performance \\ 
% Supplementary Materials} \par}
%   \vskip 0.25in
% \hrule height 1pt
% \vskip 0.09in
% }
% \input{appendix}
% \section{Concluding remark}

\end{document}

%% file: header.tex
\newcommand{\Eset}{\mathbb{E}}

\newcommand{\Pset}{\mathbb{P}}

\newcommand{\Rset}{\mathbb{R}}

\newcommand{\Kcal}{{\cal K}}

\newcommand{\Ocal}{{\cal O}}

\newcommand{\Qcal}{{\cal Q}}

\newcommand{\Abf}{{\bf A}}

\newcommand{\xhat}{{\hat{x}}}
\newcommand{\yhat}{{\hat{y}}}
\newcommand{\zhat}{{\hat{z}}}

\newtheorem{lem}{Lemma}
\newtheorem{thm}{Theorem}

\newtheorem{assump}{Assumption}
\newtheorem{remark}{Remark}

%% file: intro.tex
%!TEX root = local_SA.tex

\section{Nonlinear two-time-scale SA}\label{sec:nonlinear_SA}
Stochastic approximation ({SA}), introduced by \cite{RobbinsM1951}, is a simulation-based approach for finding the root (or fixed point) of some unknown operator $F$ represented by the form of an expectation, i.e., $F(x) = \Eset_{\pi}[F(x,\xi)]$, where $\xi$ is some random variable with a distribution $\pi$. Specifically, this method seeks a point $x^{\star}$ such that $F(x^{\star}) = 0$ based on the noisy observations $F(x;\xi)$. The iterate $x$ is iteratively updated by moving along the direction of $F(x;\xi)$ scaled by some step size. Through a careful choice of this step size, the ``noise'' induced by the random samples $\xi$ can be averaged out across iterations, and the algorithm converges to $x^*$.  {SA} has found broad applications in many areas including statistics, stochastic optimization, machine learning, and reinforcement learning \cite{borkar2008,SBbook2018,LanBook2020}.        

In this paper, we consider the two-time-scale {SA}, a generalized variant of the classic {SA}, which is used to find the root of a system of two coupled nonlinear equations. Given two unknown operators $F:\Rset^{d}\times\Rset^{d}\rightarrow\Rset^{d}$ and $G:\Rset^{d}\times\Rset^{d}\rightarrow\Rset^{d}$ represented by 
\begin{align*}
F(x,y) = \Eset_{\pi}[F(x,y;\xi)]\quad \text{and}\quad G(x,y) = \Eset_{\pi}[G(x,y;\xi)],
\end{align*}
we seek to find $x^{\star}$ and $y^{\star}$ such that
\begin{align}
\left\{
\begin{aligned}
&F(x^{\star},y^{\star}) = 0\\
&G(x^{\star},y^{\star}) = 0.
\end{aligned}\right.\label{prob:FG}
\end{align}
Since $F$ and $G$ are unknown, we assume that there is a stochastic oracle that outputs noisy values of $F(x,y)$ and $G(x,y)$ for a given pair $(x,y)$, i.e., we only have access to $F(x,y;\xi)$ and $G(x,y;\xi)$. Using this stochastic oracle, we study the two-time-scale nonlinear {SA} for solving problem \eqref{prob:FG}, which iteratively updates the iterates $x_{k}$ and $y_{k}$, the estimates of $x^{\star}$ and $y^{\star}$, respectively, for any $k\geq 0$ as
\begin{align}
\begin{aligned}
x_{k+1} &= x_{k} - \alpha_{k}F(x_{k},y_{k};\xi_{k})\\   
y_{k+1} &= y_{k} - \beta_{k}G(x_{k},y_{k};\xi_{k}),
\end{aligned}\label{alg:xy}
\end{align}
where $x_{0}$ and $y_{0}$ are arbitrarily initial conditions and $\{\xi_{k}\}$ is a sequence of random variables. We consider the case where $\{\xi_{k}\}$ is a Markov chain, whose stationary distribution is $\pi$. Thus,  ${\xi_{k}}$ are dependent and the observations are biased, i.e., 
\begin{align*}
\Eset_{\xi_{k}}[F(x;y;\xi_{k})] \neq F(x,y) \;\text{ and }\;    \Eset_{\xi_{k}}[G(x;y;\xi_{k})] \neq G(x,y).
\end{align*}

In \eqref{alg:xy}, $\alpha_{k}$ and $\beta_{k}$ are two nonnegative step sizes chosen such that $\beta_{k}\ll \alpha_{k}$, i.e., the second iterate is updated using step sizes that are very small as compared to the ones used to update the first iterate.  Thus, the update of $x_{k}$ is referred to as the ``fast-time scale" while the update of $y_{k}$ is called the ``slow-time scale". The time-scale difference here is loosely defined as the ratio between the two step sizes, i.e., $\beta_{k}/\alpha_{k}$. In addition, the update of the \textit{fast iterate} depends on the \textit{slow iterate} and vice versa, that is, they are coupled to each other. To handle this coupling, the two step sizes have to be properly chosen to guarantee the convergence of the method. Indeed, an important problem in this area is to select the two step sizes so that the two iterates converge as fast as possible. Our main focus is, therefore, to derive the finite-time convergence of \eqref{alg:xy} in solving \eqref{prob:FG} under some proper choice of these two step sizes and to understand their impact on the performance of the nonlinear two-time-scale {SA} under Markovian randomness caused by the Markov process $\{\xi_{k}\}$. 

\input{motivation}

\subsection{Main contributions}
The focus of this paper is to derive the finite-time performance of the nonlinear two-time-scale {SA} under Markov randomness. In particular, under some proper choice of step sizes $\alpha_{k}$ and $\beta_{k}$, we show that the method achieves a convergence in expectation at a rate $\Ocal(\log(k)/k^{2/3})$, where $k$ is the number of iterations. Our convergence rate is similar to the ones in the i.i.d setting except for the log factor, which captures the mixing rate of the  Markov chain. Our analysis is mainly motivated by the classic singular perturbation theory for studying the asymptotic convergence of two-time-scale systems, that is, we consider a Lyapunov function that carefully characterizes the coupling between the two iterates. In addition, we utilize the geometric mixing time of the underlying Markov process to handle the bias and dependence in the data.

\subsection{Related works}
\input{related_works}

%% file: motivation.tex
%!TEX root = nonlinear_SA.tex

\subsection{Motivating applications}\label{sec:motivating_applications}
Nonlinear two-time-scale {\sf SA}, Eq.\ \eqref{alg:xy}, has found numerous applications in many areas including  stochastic optimization  \cite{WangFL2017, ZhangX2019}, distributed control over cluster networks \cite{ThiemDN2020}, distributed optimization under communication constraints  \cite{DoanBS2017,DoanMR2018b},   and reinforcement learning \cite{BORKAR2005,KondaT2003}. The Markov setting we consider in this paper can be found in the applications where the generated data are dependent and evolve through time, for example, they are sampled from some dynamical systems. Notable examples include robust estimation \cite{POLJAK198053},  stochastic control/reinforcement learning \cite{BORKAR2005,KondaT2003}, Markov chain Monte Carlo methods \cite{Andrieu2003}, and (distributed) incremental stochastic optimization \cite{RamNV2009,JohanssonRJ2010}. We provide below two such motivating applications.

One concrete example is to model different variants of the well-known stochastic gradient descent (SGD) where the data is generated by a Markov process such as in robust estimation \cite{POLJAK198053}. In this problem, we assume that the data is generated by an autoregressive process, that is, the data points $\xi_{k} = (\xi_{k}^{1},\xi_{k}^{2})\in\Rset^{d}\times\Rset$ is generated as follows
\begin{align*}
\xi_{k}^{1} = \Abf\xi_{k-1}^{1} + e_{1}W_{k},\quad \xi_{k}^{2} = \langle x,\xi_{k}^{1}\rangle + V_{k}, 
\end{align*}
where $e_{1}$ is the first basis vector, and $W_{k}$ and $V_{k}$ are sampled i.i.d from the standard normal distribution $N(0,1)$. In addition, $\Abf\in\Rset^{d\times d}$ is a subdiagonal matrix, where $\Abf_{i,i-1}$ is drawn uniformly from $[.8,.99]$. Obviously, since $\{W_{k}\}$ and $\{V_{k}\}$ are i.i.d $\{\xi_{k}\}$ is a Markov chain. The objective of robust identification is to estimate the system parameter $x$ from these Markov samples. In robust identification problems, we want to find $x$ that optimizes 
\begin{align}
\underset{x\in\Rset^{d}}{\text{minimize }} f(x) = \Eset_{\pi} \left[F(x;\xi)\right],
\end{align}
where $F$ is some lost function, e.g., $F(x;\xi) = \left(\langle x,\xi^{1}\rangle - \xi^{2}\right)^2$. For solving this problem, we consider SGD with the Polyak-Rupert averaging, where an additional averaging iterate is used to improve the performance the classic SGD \cite{PolyakJ1992,Ruppert88}
\begin{align*}
y_{k+1} &= y_{k} - \beta_{k} \nabla f(y_{k};\xi_{k}),\\
x_{k+1} &= \frac{1}{k+1}\sum_{t=0}^{k}y_{k} = x_{k} + \frac{1}{k+1}\left(y_{k} - x_{k}\right),
\end{align*}
which is a special form of \eqref{alg:xy}. Other applications of SGD under Markov samples can be found in incremental optimization \cite{RamNV2009,JohanssonRJ2010}, where the iterates are updated based on a finite Markov chain. In general, it has been shown in \cite{SunSY2018} that using Markov samples also help to improve the performance of SGD as compared to the case of i.i.d samples. In this case, SGD with Markov samples, while using less data and computation, converges faster than the i.i.d counterpart.  

As another example, two-time-scale {\sf SA} has been used extensively to model  reinforcement learning methods, for example, gradient temporal difference (TD) learning  and actor-critic methods \cite{Maeietal2009,KondaT2003,wu_actor_critic2020,Hong_actor_critic2020,Khodadadian_2021} and the references therein. In reinforcement learning, problems are often modeled as Markov decision processes, therefore, Markov samples are a natural setting. To be specific, we consider the gradient {\sf TD} learning for solving the policy evaluation problem under nonlinear function approximations studied in \cite{Maeietal2009}, which can be viewed as a variant of \eqref{alg:xy}. In this problem, we want to estimate the cumulative rewards $V$ of a stationary policy using function approximations $V_{y}$, that is, our goal is to find $y$ so that $V_{y}$ is as close as possible to the true value $V$. Here, $V_{y}$ can be represented by a neural network where $y$ is the weight vector of the network. Let $\zeta$ be the environmental sate, whose transition is governed by a Markov process. In addition, let $\gamma$ be the discount factor, $\phi(\zeta) = \nabla V_{y}(\zeta)$ be the feature vector, and $r$ be the reward returned by the environment. Given a sequence of samples $\{\zeta_{k},r_{k}\}$, one version of  {\sf GTD} is
\begin{align*}
x_{k+1} &= x_{k} + \alpha_{k}( \delta_{k} - \phi(\zeta_{k})^{T}x_{k})\phi(\zeta_{k})\\
y_{k+1} &= y_{k} + \beta_{k}\Big[\left(\phi(\zeta_{k})-\gamma\phi(\zeta_{k+1})\right)\phi(\zeta_{k})^Tx_{k} - h _{k}\Big],
\end{align*}
where $\delta_{k}$ and $h_{k}$ are defined as  
\begin{align*}
\delta_{k} &= r_{k} + \gamma V_{y_{k}}(\zeta_{k+1}) - V_{y_{k}}(\zeta_{k}),\\
h_{k} &= (\delta_{k} - \phi(\zeta_{k})^Tx_{k})\nabla^2V_{y_{k}}(\zeta_{k})x_{k},    
\end{align*}
which is clearly a variant of \eqref{alg:xy} under some proper choice of $F$ and $G$. It has been observed that gradient {\sf TD} is more stable and performs better compared to the single-time-scale counterpart ({\sf TD} learning) under off-policy learning and nonlinear function approximations \cite{Maeietal2009}.

%% file: related_works.tex
Given the broad applications of {\sf SA} in many areas, its convergence properties have received much interests for years. In particular, the asymptotic convergence of {\sf SA}, including its two-time-scale variant, can be established by using the (almost) Martingale convergence theorem  when the noise are i.i.d or the ordinary differential equation ({\sf ODE}) method for more general noise settings; see for example \cite{borkar2008,benveniste2012adaptive}. Under the right conditions both of these methods show that the noise effects eventually average out and the {\sf SA} iterate asymptotically converges to the desired solutions. 

The convergence rate of the single-time-scale {\sf SA} has been studied extensively for years under different settings due to its broad applications in machine learning and stochastic optimization. The asymptotic rate of this method can be studied by using the Central Limit Theorem ({\sf CLT}), but requiring substantially stronger assumptions \cite{borkar2008,KY2009,benveniste2012adaptive}. On the other hand, the finite-time bounds of {\sf SA} has been studied under both i.i.d and Markov settings; see for example \cite{Chen_SA_Envelope_2020,Qu_SyncSA_202,Mou_SA_2020,Qu_SyncSA_202,Karimi_colt2019, SrikantY2019_FiniteTD,Chen_MC_LinearSA_2020,ChenZDMC2019} and the references therein. We also note that there are also different work on studying the finite-time performance of {\sf SA} in the context of {\sf SGD} both in the i.i.d and Markovian noise models; see for example \cite{BottouCN2018, SunSY2018,DoanNPR2020a,Nagaraj_Acc_LS_NeurIPS2020} and the references therein.

Unlike the single-time-scale {\sf SA}, the convergence rates of the two-time-scale {\sf SA} are less understood due to the complicated interactions between the two step sizes and the iterates.  Specifically, the rates of the two-time-scale {\sf SA} has been studied mostly for the linear settings in both i.i.d and Markovian settings, i.e, when $F$ and $G$ are linear functions w.r.t their variables; see for example in \cite{KondaT2004, DalalTSM2018, Dalal_Szorenyi_Thoppe_2020, DoanR2019,GuptaSY2019_twoscale,Doan_two_time_SA2019,Kaledin_two_time_SA2020,Dalal_Szorenyi_Thoppe_2020}. For the nonlinear settings, we are only aware of the work in \cite{MokkademP2006,Doan_two_time_SA2020}, which considers the convergence rates of the nonlinear two-time-scale {\sf SA} in \eqref{alg:xy} under i.i.d settings. In particular, under the stability condition (Assumption $1$ in \cite{MokkademP2006}, $lim_{k\rightarrow\infty}(x_{k},y_{k}) = (x^{\star},y^{\star}))$ and when $F$ and $G$ can be locally approximated by linear functions in a neighborhood of $(x^{\star},y^{\star})$, a convergence rate of \eqref{alg:xy} in distribution is provided in \cite{MokkademP2006}. This work also shows that the rates of the fast-time and slow-time scales are asymptotically decoupled under proper choice of step sizes, which agrees with the previous observations of linear two-time-scale {\sf SA}; see for example \cite{KondaT2004}. On the other hand, the work in \cite{Doan_two_time_SA2020} studies the finite-time bound of \eqref{alg:xy} under different assumptions on the operators $F$ and $G$ as compared to the ones considered in \cite{MokkademP2006}. The setting considered in this paper is similar to the ones studied in \cite{Doan_two_time_SA2020}, explained in detail in Section \ref{sec:results}. However, unlike the work in \cite{MokkademP2006} and \cite{Doan_two_time_SA2020}, we study the finite-time performance of \eqref{alg:xy} under Markov randomness, where we consider different techniques as compared to the ones in \cite{MokkademP2006,Doan_two_time_SA2020} due to the dependence and bias of the observations in our updates. More details are discussed in the next section.

% \tdoan{Discuss about the asymptotic rate and the finite-time bounds}. 

% In this paper, our focus is to study the finite-time analysis that characterizes the rates of \eqref{alg:xy} in mean square errors. We do this under different  for example, we do not require the stability condition. Our setting is motivated by the conditions considered in \cite[Chapter 7]{Kokotovic_SP1999}, where the authors study the continuous-time and deterministic version of \eqref{alg:xy}, i.e., $\xi_{k} = \psi_{k} = 0$.     

%% file: results.tex
%!TEX root = nonlinear_SA.tex

\section{Main Results}\label{sec:results}
In this section, we present in detail the main results of this paper, that is, we provide a finite-time analysis for the convergence rates of \eqref{alg:xy} in mean square errors. Under some certain conditions explained below, we show that the mean square errors converge to zero at a rate
\begin{align*}
\Eset\left[\|y_{k}-y^{\star}\|^2\right] + \frac{\beta_{k}}{\alpha_{k}}\Eset\left[\|x_{k}-x^{\star}\|^2\right] \leq \Ocal\left(\frac{1}{(k+1)^2} + \frac{\log(k+1)}{(k+1)^{2/3}}\right),    
\end{align*}
where the choice of $\beta_{k}\ll\alpha_{k}$ will be discussed explicitly in the next section. We note that this convergence rate is the same as the one in the i.i.d settings \cite{Doan_two_time_SA2020}, except for the log factor that captures the mixing time of the underlying Markov chain $\{\xi_{k}\}$. To derive our theoretical result, in the next two subsections we present main technical assumptions and preliminaries used in our analysis.

\subsection{Main Assumptions}
We first present the main technical assumptions used to derive our finite-time convergence results. First, we discuss the assumptions on the operators $F$ and $G$, which are motivated by the ones required to establish the stability of the corresponding deterministic two-time-scale differential equations of \eqref{alg:xy} in \cite{Kokotovic_SP1999}. For an ease of exposition, we assume here that $(x^{\star},y^{\star}) = (0,0)$. Since $\beta_{k}\ll\alpha_{k}$ the update of $x_{k}$ is referred to as the ``fast-time" scale while $y_{k}$ is updated at a ``slow-time" scale. The time-scale difference between these two updates is loosely defined by the ratio $\beta_{k}/\alpha_{k}\ll1$, which is equivalent to the one in \cite{Kokotovic_SP1999}. To further present our motivation we consider the case of constant step sizes, i.e., $\alpha_{k} = \alpha$ and $\beta_{k} = \beta \ll \alpha$ for some proper chosen constants $\alpha,\beta$. Under appropriate choice of step sizes and proper conditions on the noise sequence $\xi_{k}$, the ODE method shows that the asymptotic convergence of the iterates in \eqref{alg:xy} is equivalent to the stability the following differential equations \cite{borkar2008}
\begin{align}
\begin{aligned}
\dot{x} = \frac{dx}{dt} &= -F(x(t),y(t))\\   
\dot{y} = \frac{dy}{dt} &= -\frac{\beta}{\alpha} G(x(t),y(t)).
\end{aligned}\label{alg:xy_ODE}
\end{align} 
First, since $\beta/\alpha \ll 1$, $y(t)$ is updated much slower than $x(t)$, therefore, one can view $y(t)$ being static in $\dot{x}$. By fixing $y(t) = y$ we have 
\begin{align}
\frac{d x}{dt} = - F(x(t),y).\label{motivation:dotx}      
\end{align}
Thus, to study the stability of $x(t)$ one needs at least to guarantee that this ODE equation has a solution for any given $y$. In this case, the equilibrium of \eqref{motivation:dotx} is a function of $y$, i.e., there exists some operator $H$ such that  $F(H(y),y) = 0$. A standard condition to guarantee the existence of an equilibrium of \eqref{motivation:dotx} is that the operators $H$ and $F$ are Lipschitz continuous as stated in the following assumption.

%%%%%%%%%%%%%%
\begin{assump}\label{assump:smooth:FH}
Given $y\in\Rset^{d}$ there exists an operator $H:\Rset^{d}\rightarrow\Rset^{d}$ such that $x = H(y)$ is the unique solution of
\begin{align*}
F(H(y),y) = 0,    
\end{align*}
where $H$ and $F$ are Lipschitz continuous with constant $L_{H}$ and $L_{F}$, respectively, i.e., $\forall x_{1}, x_{2}, y_{1}, y_{2} \in\Rset^{d}$
\begin{align}
&\hspace{-0.3cm}\|H(y_{1}) - H(y_{2})\| \leq L_{H}\|y_{1}-y_{2}\|,\label{assump:smooth:FH:ineqH}\\
&\hspace{-0.3cm}\|F(x_{1},y_1) - F(x_{2},y_2)\| \leq L_{F}(\|x_{1}-x_{2}\| + \|y_{1} - y_{2}\|).     \label{assump:smooth:FH:ineqF}
\end{align}
\end{assump}
\begin{remark}
In the case of linear two-time-scale SA, i.e., $F$ and $G$ are linear 
\begin{align*}
F(x,y) &= \Abf_{11}x + \Abf_{12}y,\\
G(x,y) &= \Abf_{21}x + \Abf_{22}y,
\end{align*}
where $\Abf_{11}$ is negative definite (but not necesarily symmetric) \cite{KondaT2004,Doan_two_time_SA2019}. We then have  $H(y) = -\Abf_{11}^{-1}\Abf_{12}y$ is a linear operator . 
\end{remark}
Second, for the global asymptotic convergence of $x(t)$ to the equilibrium of \eqref{motivation:dotx}  it is necessary that this equilibrium is unique. This condition is guaranteed if $F$ is strong monotone. 

\begin{assump}\label{assump:sm:F}
$F$ is strongly monotone w.r.t $x$  when $y$ is fixed, i.e., there exists a constant $\mu_{F} > 0$ 
\begin{align}
\left\langle x - z, F(x,y) - F(z,y) \right\rangle \geq \mu_{F} \|x-z\|^2. \label{assump:sm:F:ineq}    
\end{align}
\end{assump}
%%%%%%%%%%%%%%%
These two assumptions are also considered under different variants in the context of both linear and nonlinear two-time-scale {\sf SA} studied in \cite{KondaT2004, DalalTSM2018, DoanR2019,GuptaSY2019_twoscale,Doan_two_time_SA2019,Kaledin_two_time_SA2020,MokkademP2006}. Similarly, once $x(t)$ converges to $H(y(t))$ the convergence of $y(t)$ can be shown through studying the stability of the following differential equation
\begin{align}
\frac{dy}{dt} = -\frac{\beta}{\alpha} G(H(y(t)),y(t)).    \label{motivation:doty}
\end{align}
We again require two similar assumptions to guarantee the existence and uniqueness of the solution of \eqref{motivation:doty}. 
%%%%%%%%%%%%%%%%%%%%%%%%%
\begin{assump}\label{assump:G}
The operator $G(\cdot,\cdot)$ is Lipschitz continuous with constant $L_{G}$, i.e., $\forall x_{1}, x_{2}, y_{1}, y_{2} \in\Rset^{d}$,
\begin{align}
\hspace{-0.3cm}\|G(x_{1},y_{1}) - G(x_{2},y_{2})\| \leq L_{G}\left(\|x_{1} - x_{2}\| + \|y_{1} - y_{2}\| \right).    \label{assump:G:smooth}
\end{align}
Moreover, $G$ is $1$-point strongly monotone w.r.t $y^{\star}$, i.e., there exists a constant $\mu_{G} > 0$ such that for all $y\in\Rset^{d}$
\begin{align}
\left\langle y - y^{\star}, G(H(y),y) \right\rangle \geq \mu_{G} \|y - y^{\star}\|^2. \label{assump:G:sm}    
\end{align}
\end{assump}
Assumptions \ref{assump:smooth:FH}--\ref{assump:G} are used in \cite[Chapter 7]{Kokotovic_SP1999} to study the globally asymptotic stability of \eqref{alg:xy_ODE}. Our focus is on the finite-time convergence of the stochastic variant \eqref{alg:xy} of \eqref{alg:xy_ODE}. We, therefore, require the following assumption on the Lipschitz continuity of $F$ and $G$. 
\begin{assump}\label{assump:smooth:samples}
Given any $\xi$, the operators $F(\cdot,\cdot;\xi)$ and $G(\cdot,\cdot,\xi)$ are Lipschitz continuous with constant $L_{F}$ and $L_{G}$, respectively. That is, for any $x_{1},x_{2},y_{1},y_{2}\in\Rset^{d}$ we have
\begin{align}
    \begin{aligned}
&\|F(x_{1},y_{1};\xi) - F(x_{2},y_{2};\xi)\| \leq L_{F}(\|x_{1} - x_{2}\| + \|y_{1}-y_{2}\|)\quad \text{a.s.,}\\
&\|G(x_{1},y_{1};\xi) - G(x_{2},y_{2};\xi)\| \leq L_{G}(\|x_{1} - x_{2}\| + \|y_{1}-y_{2}\|)\quad \text{a.s.}
    \end{aligned}
\end{align} \label{assump:smooth:samples:ineq}
\end{assump}
Note that assumption \ref{assump:smooth:samples} is weaker than the boundedness condition on $F$ and $G$. Under this assumption, the iterates $\{x_{k},y_{k}\}$ can be potentially unbounded. 

Finally, we present the assumption on the noise model, which basically states that the Markov chain $\{\xi_{k}\}$ has geometric mixing time. In particular, we denote by $\tau(\alpha)$ the mixing time of $\{\xi_{k}\}$ associated with a positive constant $\alpha$, which basically tells us how long the Markov chain gets close to its stationary distribution $\pi$ \cite{LevinPeresWilmer2006}. The following assumption formally states the condition of $\tau(\alpha)$.
\begin{assump}\label{assump:mixing}
The sequence $\{\xi_{k}\}$ is a Markov chain with a compact state space $\Xi$ and has stationary distribution $\pi$. For all $x,y\in\Rset^{d}$ and $\xi\in\Xi$ and a given $\alpha > 0$ we have $\forall  k\geq \tau(\alpha)$
\begin{align}
\begin{aligned}
&\left\|\Eset[F(x,y;\xi_{k})] - F(x,y)\,|\, \xi_{0} = \xi\right\| \leq \alpha\\  
&\left\|\Eset[G(x,y;\xi_{k})] - G(x,y)\,|\, \xi_{0} = \xi\right\| \leq \alpha.
\end{aligned}\label{assump:mixing:ineq}
\end{align}
Moreover, $\{\xi_{k}\}$ has a geometric mixing time, i.e., there exists a positive constant $C$ such that \begin{align}
\tau(\alpha) = C\log\left(\frac{1}{\alpha}\right).\label{assump:mixing:tau}
\end{align} 
\end{assump} 
Assumption \ref{assump:mixing} basically sates that given $\alpha>0$ there exists $C>0$ s.t. $\tau(\alpha) = C\log(1/\alpha)$ and
\begin{align}
 \|\Pset^{k}(\xi_{0},\cdot) - \pi \|_{TV}  \leq \alpha,\quad \forall k\geq \tau(\alpha),\;\forall \xi_{0}\in\Xi,\label{notation:tau}    
\end{align}
where  $\|\cdot\|_{TV}$ is the total variance distance and $\Pset^{k}(\xi_{0},\xi)$ is the probability that $\xi_{k} = \xi$ when we start from $\xi_{0}$ \cite{LevinPeresWilmer2006}. This assumption holds in various applications, e.g, in incremental optimization \cite{RamNV2009,JohanssonRJ2010}, where the iterates are updated based on a finite Markov chain, and in reinforcement learning problems with  a finite number states \cite{silver2017mastering}.  Assumption \ref{assump:mixing} is used in the existing literature to study the finite-time performance of {\sf SGD} and SA under Markov randomness; see  \cite{SunSY2018, SrikantY2019_FiniteTD, Kaledin_two_time_SA2020} and the references therein. Finally, since $\Xi$ is compact the Lipschitz continuity of $F$ and $G$ also gives the following result. 
\begin{lem}\label{lem:FG_bounded}
Let $B$ be a constant defined as  
\begin{align}
B = \max\{\max_{\xi\in\Xi}\{\|F(0,0,\xi)\|, \|G(0,0,\xi)\|\},\|F(0,0)\|, \|G(0,0)\|, L_{F},L_{G},L_{H}\}.\label{lem:FG_bounded:B}
\end{align}
Then for all $x,y$ we have
\begin{align}
\begin{aligned}
&\max\{\|F(x,y,\xi)\|,\|F(x,y)\|\} \leq B(\|x\| + \|y\| + 1)\quad \text{a.s.},\\    
&\max\{\|G(x,y,\xi)\|,\|G(x,y)\|\} \leq B(\|x\| + \|y\| + 1)\quad \text{a.s.}
\end{aligned}\label{lem:FG_bounded:ineq}
\end{align}
\end{lem}

\subsection{Preliminaries}
We now present some preliminaries, which will be useful to derive our main result studied in the next subsection. For an ease of exposition, we present the analysis of these results in Section \ref{sec:lemmas_proofs}. 

To study the performance of SA one can analyze the convergence rate of the following mean square error
\begin{align*}
\Eset[\|x_{k}-x^\star\|^2+\|y_{k}-y^\star\|^2]\quad \text{to zero}.     
\end{align*}
However, this mean square error does not explicitly characterize  the coupling between the two iterates. We, therefore, consider a different notion of mean square error, which will help us to facilitate our development. In particular, under Assumption \ref{assump:smooth:FH} and by \eqref{prob:FG} we have $x^{\star} = H(y^{\star})$ and 
\begin{align*}
F(H(y^{\star}),y^{\star}) = 0\quad\text{and}\quad G(H(y^{\star}),y^{\star}) = 0.  
\end{align*}
The coupling between $x$ and $y$ is represented through $H$, motivated us to consider the following two residual variables
\begin{align}
\begin{aligned}
\xhat_{k} &= x_{k} - H(y_{k})\\   
\yhat_{k} &= y_{k} - y^{\star}.
\end{aligned}    \label{alg:xyhat}
\end{align}
Obviously, if $\yhat_{k}$ and $\xhat_{k}$ go to zero, $(x_{k},y_{k})\rightarrow(x^{\star},y^{\star})$. Thus, to establish the convergence of $(x_{k},y_{k})$ to $(x^{\star},y^{\star})$ one can instead study the convergence of $(\xhat_{k},\yhat_{k})$ to zero. The rest of this paper is to focus on deriving the convergence rates of these variables to zero.

For our analysis, we consider nonincreasing and nonnegative time-varying sequences of step sizes $\{\alpha_{k},\beta_{k}\}$ satisfying 
\begin{align}
\sum_{k=0}^{\infty}\alpha_{k} = \sum_{k=0}^{\infty}\beta_{k} = \infty\quad\text{and}\quad \sum_{k=0}^{\infty}\Big(\alpha_{k}^2 + \beta_{k}^2 + \frac{\beta_{k}^2}{\alpha_{k}}\Big) < \infty,      \label{notation:alpha_beta}
\end{align} 
which also implies that $\beta_{k}\ll \alpha_{k}$ and $\lim_{k\rightarrow \infty}\alpha_{k} = \lim_{k\rightarrow \infty}\beta_{k} = 0$. Since $\alpha_{k}$ decreases to zero and $\tau(\alpha_{k}) = \log(1/\alpha_{k})$ there exists a positive integer $\Kcal^{\star}$ s.t. 
\begin{align}\label{notation:K*}
\alpha_{k;\tau(\alpha_{k})} \triangleq\sum_{t=k-\tau(\alpha_{k})}^{k}\alpha_{t} \leq \tau(\alpha_{k})\alpha_{k-\tau(\alpha_{k})} \leq \min\left\{\frac{\log(2)}{2B},\, \alpha_{0}\right\},\quad \forall k \geq\Kcal^{\star}. 
\end{align} For convenience, we denote by 
\begin{align}
\begin{aligned}
\psi_{k} &= F(x_{k},y_{k};\xi_{k}) - F(x_{k},y_{k}),\\
\zeta_{k} &= G(x_{k},y_{k};\xi_{k}) - G(x_{k},y_{k}),
\end{aligned}\label{notation:psi_zeta}
\end{align}
so \eqref{alg:xy} can be rewritten as
\begin{align}
\begin{aligned}
x_{k+1} &= x_{k} - \alpha_{k}(F(x_{k},y_{k}) + \psi_{k})\\   
y_{k+1} &= y_{k} - \beta_{k}(G(x_{k},y_{k}) + \zeta_{k}),
\end{aligned}\label{alg:xy_noise}
\end{align}
Note that $\{\psi_{k},\zeta_{k}\}$ are Markovian, therefore, they are dependent and have mean different to zero. We denote by $\Qcal_{k}$ the filtration contains all the history generated by the algorithms upto time $k$, i.e.,
\begin{align*}
\Qcal_{k} = \{x_{0},y_{0},\xi_{0},\xi_{1},\psi_{1},\ldots,\xi_{k-1}\}.
\end{align*}
%%%%%%%%%%
%%%%%%%%%
Finally, for an ease of exposition we define 
We denote by 
\begin{align}
z = \left[\begin{array}{c}
x\\
y 
\end{array}\right],\quad \zhat = \left[\begin{array}{c}
x - H(y)\\
y - y^{\star}
\end{array}\right]= \left[\begin{array}{c}
\xhat\\
\yhat 
\end{array}\right].\label{notation:z_zhat}
\end{align}
The main challenges in our analysis are two fold: $1)$ the coupling between the fast and slow iterates and $2)$ the dependence and bias in the observations of $F$ and $G$ due to the Markov model. We handle the first challenge by introducing a proper weighted Lyapunov function that combines the norms of $\xhat_{k}$ and $\yhat_{k}$, which we will discuss in the next section. On the other hand, we utilize the geometric mixing time to handle the Markovian noise, which is used in the following three lemmas to characterize the sizes of the two residual variables.

\begin{lem}\label{lem:xhat}
Suppose that Assumptions \ref{assump:smooth:FH}--\ref{assump:mixing} hold. Then we have for all $k\geq \Kcal^*$
\begin{align}
\Eset[\|\xhat_{k+1}\|^2]  &\leq (1-\mu_{F}\alpha_{k})\Eset[\|\xhat_{k}\|^2] + 32(1+B)^{6}\Big(\frac{5\beta_{k}^2}{\mu_{F}\alpha_{k}}  + \beta_{k}^2 + \alpha_{k;\tau(\alpha_{k})}\alpha_{k}\Big)\Eset[\|\zhat_{k}\|^2]\notag\\ 
&\quad + 32(1+B)^{6}(\|y^{\star}\| + \|H(0)\| + 1)^2\Big(\frac{5\beta_{k}^2}{\mu_{F}\alpha_{k}}  + \beta_{k}^2 + \alpha_{k}\alpha_{k;\tau(\alpha_{k})}\Big). \label{lem:xhat:ineq}
\end{align}
\end{lem}

\begin{lem}\label{lem:yhat}
Suppose that Assumptions \ref{assump:smooth:FH}--\ref{assump:mixing} hold.  Then we have for all $k\geq \Kcal^*$
\begin{align}
\Eset[\|\yhat_{k+1}\|^2]  &\leq (1- \mu_{G}\beta_{k})\Eset[\|\yhat_{k}\|^2]   + 18(1+B)^{4}(\alpha_{k}\beta_{k} + 10B\alpha_{k;\tau(\alpha_{k})}\beta_{k} + 3\beta_{k}^2)\Eset[\|\zhat_{k}\|^2] \notag\\
&\quad + 24(1+B)^{4}(\|y^{\star}\| + \|H(0)\| + 1)^2(\beta_{k}^2 + 7B\alpha_{k;\tau(\alpha_{k})}\beta_{k}) + \frac{B^2}{\mu_{G}}\beta_{k}\Eset[\|\xhat_{k}\|^2].\label{lem:yhat:ineq}
\end{align}
\end{lem}

\begin{lem}\label{lem:zhat:upperbound}
Suppose that Assumptions \ref{assump:smooth:FH}--\ref{assump:mixing} hold. Let $D_{1}$ and $D_{2}$ be defined as 
\begin{align}
\begin{aligned}
D_{1} &=  \sum_{k=0}^{\infty} \frac{\beta_{k}^2}{\mu\alpha_{k}} +  \beta_{k}^2 + \alpha_{k}\alpha_{k;\tau(\alpha_{k})} < \infty,\\
D_{2} &= 160(1+B)^{6}(\|y^{\star}\| + \|H(0)\| + 1)^2.
\end{aligned}
\label{notation:CD}
\end{align}
Then we obtain for all $k\geq \Kcal^*$
\begin{align}
\Eset[\|\zhat_{k}\|^2] &\leq D \triangleq  \Eset[\|\zhat_{0}\|^2]e^{160D_{1}(B+1)^{6}}  + D_{1}D_{2}e^{320D_{1}(B+1)^{6}}.   \label{lem:zhat:upperbound:ineq}
\end{align}
\end{lem}

\subsection{Convergence Rates}
In this section, we present the main result of this paper, which is the convergence rate of \eqref{alg:xy}. To do it, we introduce the following candidate of Lypapunov function, which takes into account the time-scale difference between these two residual variables
\begin{align}
V(\xhat_{k},\yhat_{k}) &= \Eset[\|\yhat_{k}\|^2] + \frac{2B^2}{\mu_{F}\mu_{G}}\frac{\beta_{k}}{\alpha_{k}} \Eset[\|\xhat_{k}\|^2],\label{notation:V}    
\end{align}
where $\frac{2B^2}{\mu_{F}\mu_{G}}\frac{\beta_{k}}{\alpha_{k}}$ is to characterize the time-scale difference between the two residuals. Our main result, which is the finite-time bound of the rates of the residual variables to zero in expectation, is formally stated in the following theorem.

% This Lyapunov function was used in \cite{Kokotovic_SP1999} to study the globally asymptotic stability of \eqref{alg:xy_ODE}. Due to the impact of the noise, it is not obvious, however, to guarantee that the solutions of the stochastic systems \eqref{alg:xy} will track the ones of \eqref{alg:xy_ODE} and to derive the finite-time convergence of \eqref{alg:xy}. 

\begin{thm}\label{thm:main}
Suppose that Assumptions \ref{assump:smooth:FH}--\ref{assump:mixing} hold. Let $\{x_{k},y_{k}\}$ be generated by \eqref{alg:xy} with $x_{0}$ and $y_{0}$ initialized arbitrarily. Let $\alpha_{k},\beta_{k}$ be two sequence of nonnegative and nonincreasing step sizes satisfying 
\begin{align}
\begin{aligned}
&\alpha_{k} = \frac{\alpha_{0}}{(k+1)^{2/3}},\quad \beta_{k} = \frac{\beta_{0}}{k+1},\quad\frac{\beta_{0}}{\alpha_{0}} \leq \max\left\{\frac{\mu_{F}}{2\mu_{G}}\,,\,\frac{\mu_{F}\mu_{G}}{B^2}\right\}, \qquad \beta_{0} \geq\frac{1}{\mu_{G}}. \label{thm:main:stepsize}
\end{aligned}
\end{align}
Moreover, let $C$ be defined in \eqref{assump:mixing:tau}, and $D_{2},D$ be given in Lemma \ref{lem:zhat:upperbound}. Then we have for all $k\geq\Kcal^*$
\begin{align}
V_{k+1}  &\leq \frac{(\Kcal^*)^2V_{\Kcal^*}}{(k+1)^2} + \frac{5D_{2}+64D(1+B)^{8}}{2\mu_{F}\mu_{G}}\Big(\frac{5\beta_{0}^3 + 2\mu_{F}\beta_{0}\alpha_{0}^3}{\mu_{F}\alpha_{0}^2} \frac{1}{(k+1)^{2/3}} + \frac{4C\beta_{0}\alpha_{0}\log((k+1)/\alpha_{0})}{(k+1)^{2/3}}\Big).\label{thm:main:ineq}
\end{align}

\end{thm}

\begin{proof}
For convenience, we denote by $\omega_{k}$ 
\[\omega = \frac{2B^2}{\mu_{F}\mu_{G}}\frac{\beta_{k}}{\alpha_{k}}\cdot\] 
Since $\beta_{k}/\alpha_{k}$ is nonincreasing and less than $1$, multiplying both sides of \eqref{lem:xhat:ineq} by $\omega_{k}$ we have
\begin{align*}
\omega_{k+1}\Eset[\|\xhat_{k+1}\|^2]  &\leq \omega_{k}(1-\mu_{F}\alpha_{k})\Eset[\|\xhat_{k}\|^2] + \frac{64B^2(1+B)^{6}}{\mu_{F}\mu_{G}}\Big(\frac{5\beta_{k}^3}{\mu_{F}\alpha_{k}^2}  + \beta_{k}^2 + \alpha_{k;\tau(\alpha_{k})}\beta_{k}\Big)\Eset[\|\zhat_{k}\|^2]\notag\\ 
&\quad + \frac{64B^2(1+B)^{6}}{\mu_{F}\mu_{G}}(\|y^{\star}\| + \|H(0)\| + 1)^2\Big(\frac{5\beta_{k}^3}{\mu_{F}\alpha_{k}^2}  + \beta_{k}^2 + \alpha_{k;\tau(\alpha_{k})}\beta_{k}\Big)\notag\\
&\leq (1-\mu_{G}\beta_{k})\omega_{k}\Eset[\|\xhat_{k}\|^2] + (\mu_{G}\beta_{k} - \mu_{F}\alpha_{k})\omega_{k}\Eset[\|\xhat_{k}\|^2]\notag\\ 
&\quad +   \frac{64B^2(1+B)^{6}}{\mu_{F}\mu_{G}}\Big(\frac{5\beta_{k}^3}{\mu_{F}\alpha_{k}^2}  + \beta_{k}^2 + \alpha_{k;\tau(\alpha_{k})}\beta_{k}\Big)\Eset[\|\zhat_{k}\|^2]\notag\\ 
&\quad + \frac{64B^2(1+B)^{6}}{\mu_{F}\mu_{G}}(\|y^{\star}\| + \|H(0)\| + 1)^2\Big(\frac{5\beta_{k}^3}{\mu_{F}\alpha_{k}^2}  + \beta_{k}^2 + \alpha_{k;\tau(\alpha_{k})}\beta_{k}\Big),
\end{align*}
which when adding to \eqref{lem:yhat:ineq} and using \eqref{notation:V} we obtain
\begin{align}
V_{k+1} &\leq (1-\mu_{G}\beta_{k})V_{k} + \Big(\mu_{G}\beta_{k} - \mu_{F}\alpha_{k} + \frac{B^2\beta_{k}}{\mu_{G}}\omega_{k}\Big)\omega_{k}\Eset[\|\xhat_{k}\|^2]\notag\\ 
&\quad + \frac{64B^2(1+B)^{6}}{\mu_{F}\mu_{G}}\Big(\frac{5\beta_{k}^3}{\mu_{F}\alpha_{k}^2}  + \beta_{k}^2 + \alpha_{k;\tau(\alpha_{k})}\beta_{k}\Big)\Eset[\|\zhat_{k}\|^2]\notag\\ 
&\quad + \frac{64B^2(1+B)^{6}}{\mu_{F}\mu_{G}}(\|y^{\star}\| + \|H(0)\| + 1)^2\Big(\frac{5\beta_{k}^3}{\mu_{F}\alpha_{k}^2}  + \beta_{k}^2 + \alpha_{k;\tau(\alpha_{k})}\beta_{k}\Big)\notag\\
&\quad + 18(1+B)^{4}(\alpha_{k}\beta_{k} + 10B\alpha_{k;\tau(\alpha_{k})}\beta_{k} + 3\beta_{k}^2)\Eset[\|\zhat_{k}\|^2] \notag\\
&\quad + 24(1+B)^{4}(\|y^{\star}\| + \|H(0)\| + 1)^2(\beta_{k}^2 + 7B\alpha_{k;\tau(\alpha_{k})}\beta_{k})\notag\\ 
&\leq (1-\mu_{G}\beta_{k})V_{k}  + \frac{64(1+B)^{8}}{\mu_{F}\mu_{G}}\Big(\frac{5\beta_{k}^3}{\mu_{F}\alpha_{k}^2}  + 2\beta_{k}^2 + 4\alpha_{k;\tau(\alpha_{k})}\beta_{k}\Big)\Eset[\|\zhat_{k}\|^2]\notag\\ 
&\quad + \frac{64(1+B)^{8}}{\mu_{F}\mu_{G}}(\|y^{\star}\| + \|H(0)\| + 1)^2\Big(\frac{5\beta_{k}^3}{\mu_{F}\alpha_{k}^2}  + 2\beta_{k}^2 + 4\alpha_{k;\tau(\alpha_{k})}\beta_{k}\Big)\notag\\
&\stackrel{\eqref{notation:CD}}{=} (1-\mu_{G}\beta_{k})V_{k}  + \frac{64(1+B)^{8}}{\mu_{F}\mu_{G}}\Big(\frac{5\beta_{k}^3}{\mu_{F}\alpha_{k}^2}  + 2\beta_{k}^2 + 4\alpha_{k;\tau(\alpha_{k})}\beta_{k}\Big)\Eset[\|\zhat_{k}\|^2]\notag\\ 
&\quad + \frac{5D_{2}}{2\mu_{F}\mu_{G}}\Big(\frac{5\beta_{k}^3}{\mu_{F}\alpha_{k}^2}  + 2\beta_{k}^2 + 4\alpha_{k;\tau(\alpha_{k})}\beta_{k}\Big),\label{thm:main:Eq1}
\end{align}
where the second inequality we use \eqref{thm:main:stepsize} to have
\begin{align*}
\mu_{G}\beta_{k} - \mu_{F}\alpha_{k} + \frac{B^2\beta_{k}}{\mu_{G}}\omega_{k} = \mu_{G}\beta_{k} - \mu_{F}\alpha_{k} + \frac{\mu_{F}\alpha_{k}}{2} \leq 0.    
\end{align*}
First, since $\beta_{k} = \beta_{0}/(k+1)$ and $\beta_{0} \geq 2/\mu_{G}$ we have
\begin{align*}
(k+1)^2(1 - \mu_{G}\beta_{k}) \leq (k+1)^2\frac{k-1}{k+1}\leq k^2.  
\end{align*}
Multiplying both sides of \eqref{thm:main:Eq1} by $(k+1)^2$ and using the preceding relation and  \eqref{lem:zhat:upperbound:ineq} we obtain
\begin{align}
&(k+1)^2V_{k+1} \leq k^2 V_{k} + \frac{5D_{2}+64D(1+B)^{8}}{2\mu_{F}\mu_{G}} \Big(\frac{5\beta_{k}^3}{\mu_{F}\alpha_{k}^2} + 2\beta_{k}\alpha_{k} + 4\beta_{k}\alpha_{k;\tau(\alpha_{k})}\Big)(k+1)^2\notag\\
&\quad \leq k^2 V_{k} + \frac{5D_{2}+64D(1+B)^{8}}{2\mu_{F}\mu_{G}}\Big(\frac{5\beta_{0}^3(k+1)^{1/3}}{\mu_{F}\alpha_{0}^2} + 2\beta_{0}\alpha_{0}(k+1)^{1/3} + \frac{4\beta_{0}\alpha_{0}\tau(\alpha_{k})(k+1)}{(k+1-\tau(\alpha_{k}))^{2/3}}\Big).  \label{thm:main:Eq2}  
\end{align}
By using \eqref{notation:K*} we have for all $k\geq \Kcal^*$
\begin{align*}
\frac{(k+1)}{(k+1-\tau(\alpha_{k}))}\leq 2(k+1)^{1/3}.     
\end{align*}
Thus, we obtain from \eqref{thm:main:Eq2} for all $k\geq \Kcal^*$
\begin{align*}
&(k+1)^2V_{k+1}\notag\\ 
&\quad \leq k^2 V_{k} + \frac{5D_{2}+64D(1+B)^{8}}{2\mu_{F}\mu_{G}}\Big(\frac{5\beta_{0}^3(k+1)^{1/3}}{\mu_{F}\alpha_{0}^2} + 2\beta_{0}\alpha_{0}(k+1)^{1/3} + 8\beta_{0}\alpha_{0}\tau(\alpha_{k})(k+1)^{1/3}\Big) \notag\\
&\quad \leq (\Kcal^*)^2V_{\Kcal^*} + \frac{5D_{2}+64D(1+B)^{8}}{2\mu_{F}\mu_{G}}\sum_{t=\Kcal^*}^{k}\Big(\frac{5\beta_{0}^3(t+1)^{1/3}}{\mu_{F}\alpha_{0}^2} + 2\beta_{0}\alpha_{0}(t+1)^{1/3} + 8\beta_{0}\alpha_{0}\tau(\alpha_{k})(k+1)^{1/3}\Big)\notag\\
&\quad \leq (\Kcal^*)^2V_{\Kcal^*} + \frac{5D_{2}+64D(1+B)^{8}}{2\mu_{F}\mu_{G}}\Big(\frac{5\beta_{0}^3(k+1)^{4/3}}{\mu_{F}\alpha_{0}^2} + 2\beta_{0}\alpha_{0}(k+1)^{4/3} + 6\beta_{0}\alpha_{0}\tau(\alpha_{k})(k+1)^{4/3}\Big),
\end{align*}
where we use the integral test to have
\begin{align*}
&\sum_{t=\Kcal^*}^{k}(t+1)^{1/3} \leq 1 +  \int_{t=0}^{k}(t+1)^{1/3}dt \leq \frac{1}{4} + \frac{3}{4}(k+1)^{4/3},\\
&\sum_{t=\Kcal^*}^{k}\tau(\alpha_{t})(t+1)^{1/3}\leq \int_{t=0}^{k}\tau(\alpha_{t})(t+1)^{1/3}dt\leq \frac{3}{4}\tau(\alpha_{k})(k+1)^{4/3}.
\end{align*}
Diving both sides of the equation above by (k+1) and using $\tau(\alpha_{k}) = \frac{2C}{3}\log((k+1)/\alpha_{0})$ yields \eqref{thm:main:ineq}.
\end{proof}

% \begin{remark}
% As observed from Eq.\ \eqref{thm:Eq2}, the two terms decide the rate of the two-time-scale {\sf SA} are $\beta_{k}\alpha_{k}$ and $\beta_{k}^3/\alpha_{k}^2$, which also characterize the coupling between the two iterates. Our choice of $a$ and $b$ in the theorem is to balance these two terms, i.e., we want to achieve
% \begin{align*}
% \alpha_{k}\beta_{k} = \frac{\beta_{k}^3}{\alpha_{k}^2} \Rightarrow \alpha_{k}^{3} = \beta_{k}^2 \Rightarrow b = 1,\; a =\frac{2}{3}.      
% \end{align*}
% However, one can choose different values of $a$ and $b$ for different purposes. For example, if we choose $b$ smaller, e.g., $ b= 3/4$, then one can quickly remove the impacts of the initial conditions on the performance of the method. However, this will cause a slow convergence of the variance term.  
% \end{remark}

\input{main_analysis}

%% file: main_analysis.tex
%!TEX root = nonlinear_SA.tex

\section{Proof of Main Lemmas }\label{sec:lemmas_proofs}

In this section we present the analysis of the results presented in Lemmas \ref{lem:xhat}--\ref{lem:zhat:upperbound}. We require the following technical lemmas, whose proofs are presented in the appendix. 

\begin{lem}\label{lem:xhat_bias}
Suppose that Assumptions \ref{assump:smooth:FH}--\ref{assump:mixing} hold. Then for all $k\geq \Kcal^{\star}$ we have
\begin{align}
\Eset[-\xhat_{k}^T\psi_{k}] &\leq 9(1+B)^{4}\alpha_{k}\Eset[\|\zhat_{k}\|^2] + 180B^2(1+B)^{3}\alpha_{k;\tau(\alpha_{k})}\Eset[\|\zhat_{k}\|^2] \notag\\
&\quad + 156B(1+B)^{4}\alpha_{k;\tau(\alpha_{k})}(\|y^{\star}\| + \|H(0)\| + 1)^2.\label{lem:xhat_bias:ineq}
\end{align}
\end{lem}

\begin{lem}\label{lem:yhat_bias}
Suppose that Assumptions \ref{assump:smooth:FH}--\ref{assump:mixing} hold. Then for all $k\geq \Kcal^{\star}$ we have
\begin{align}
\Eset[-\yhat_{k}^T\zeta_{k}] &\leq  9(1+B)^{4}\alpha_{k}\Eset[\|\zhat_{k}\|^2] + 180B^2(1+B)^{3}\alpha_{k;\tau(\alpha_{k})}\Eset[\|\zhat_{k}\|^2] \notag\\
&\quad + 156B(1+B)^{4}\alpha_{k;\tau(\alpha_{k})}(\|y^{\star}\| + \|H(0)\| + 1)^2.\label{lem:yhat_bias:ineq}
\end{align}
\end{lem}

\begin{lem}\label{lem:z_bound}
Suppose that Assumptions \ref{assump:smooth:FH}--\ref{assump:mixing} hold. Then for all $k\geq \Kcal^{\star}$ we have
\begin{align}
&\|z_{k} - z_{k-\tau(\alpha_{k})}\| \leq 4B\alpha_{k;\tau(\alpha_{k})}(\|z_{k-\tau(\alpha_{k})}\| + 1).\label{lem:z_bound:ineq1}\\
&\|z_{k} - z_{k-\tau(\alpha_{k})}\| \leq  12B\alpha_{k;\tau(\alpha_{k})}(\|z_{k}\| + 1).\label{lem:z_bound:ineq2}
\end{align}
\end{lem}

\begin{lem}\label{lem:zhat_bound}
Suppose that Assumptions \ref{assump:smooth:FH}--\ref{assump:mixing} hold. Then for all $k\geq \Kcal^{\star}$ we have
\begin{align}
&\|\zhat_{k}-\zhat_{k-\tau(\alpha_{k})}\| \leq 4B(1+B)^2\alpha_{k;\tau(\alpha_{k})}\big(\|\zhat_{k-\tau(\alpha_{k})}\| + \|y^{\star}\| + \|H(0)\|+1\big).\label{lem:zhat_bound:ineq1}\\
&\|\zhat_{k}-\zhat_{k-\tau(\alpha_{k})}\| \leq 12B(1+B)^2\alpha_{k;\tau(\alpha_{k})}\big(\|\zhat_{k}\|+ \|y^{\star}\| + \|H(0)\|+1\big).\label{lem:zhat_bound:ineq2}\\
&\|\zhat_{k}-\zhat_{k-\tau(\alpha_{k})}\|^2 \leq 288B^2(1+B)^4\alpha_{k;\tau(\alpha_{k})}^2 \big(\|\zhat_{k}\|^2 +(\|y^{\star}\| + \|H(0)\|+1)^2\big).\label{lem:zhat_bound:ineq3}
\end{align}
\end{lem}

\begin{lem}\label{lem:psi_zeta_bound}
Let Assumptions \ref{assump:smooth:FH}--\ref{assump:G:smooth} hold. Then we have
\begin{align}
\begin{aligned}
\|\psi_{k}\| &\leq 2B(1+B)(\|\zhat_{k}\| +\|y^{\star}\| + \|H(0)\| + 1),\\
\|\zeta_{k}\| &\leq 2B(1+B)(\|\zhat_{k}\| +\|y^{\star}\| + \|H(0)\| + 1).
\end{aligned}\label{lem:psi_zeta_bound:ineq}
\end{align}
\end{lem}

\input{prelim}

%% file: prelim.tex
%!TEX root = nonlinear_SA.tex

%%%%%%%%%%
%%%%%%%%%
% \begin{lem}\label{lem:xhat}
% Suppose that Assumptions \ref{assump:smooth:FH}--\ref{assump:mixing} hold. Then we have
% \begin{align}
% \Eset[\|\xhat_{k+1}\|^2]  &\leq (1-\mu_{F}\alpha_{k})\Eset[\|\xhat_{k}\|^2] + 32(1+B)^{6}\Big(\frac{5\beta_{k}^2}{\mu_{F}\alpha_{k}}  + \beta_{k}^2 + \alpha_{k;\tau(\alpha_{k})}\alpha_{k}\Big)\Eset[\|\zhat_{k}\|^2]\notag\\ 
% &\quad + 32(1+B)^{6}(\|y^{\star}\| + \|H(0)\| + 1)^2\Big(\frac{5\beta_{k}^2}{\mu_{F}\alpha_{k}}  + \beta_{k}^2 + \alpha_{k}\alpha_{k;\tau(\alpha_{k})}\Big). \label{lem:xhat:ineq}
% \end{align}
% \end{lem}

\subsection{Proof of Lemma \ref{lem:xhat}}
\begin{proof}
Recall that $\xhat_{k} = x_{k} - H(y_{k})$. By \eqref{alg:xyhat} and \eqref{alg:xy_noise} we have
\begin{align}
\|\xhat_{k+1}\|^2 &= \|x_{k+1} - H(y_{k+1})\|^2 = \|x_{k} - \alpha_{k} (F(x_{k},y_{k}) +\psi_{k}) - H(y_{k+1})\|^2\notag\\
&= \|\xhat_{k} - \alpha_{k}F(x_{k},y_{k})  + (H(y_{k})-H(y_{k+1})) - \alpha_{k}\psi_{k}\|^2\notag\\ 
&= \|\xhat_{k} - \alpha_{k}F(x_{k},y_{k})\|^2 + \|H(y_{k})-H(y_{k+1}) - \alpha_{k}\psi_{k}\|^2  \notag\\ 
&\quad + 2(\xhat_{k} - \alpha_{k}F(x_{k},y_{k}))^T(H(y_{k})-H(y_{k+1}))  - 2\alpha_{k}\psi_{k}^T(\xhat_{k} - \alpha_{k}F(x_{k},y_{k})).\label{lem:xhat:eq1} 
\end{align}
We next analyze each term on the right-hand side of \eqref{lem:xhat:eq1}. First, 
since  $F(H(y_{k}),y_{k}) = 0$ (by Assumption \ref{assump:smooth:FH}) and by \eqref{assump:sm:F:ineq} and \eqref{lem:FG_bounded:B} we consider 
\begin{align}
&\left(\xhat_{k} - \alpha_{k}F(x_{k},y_{k})\right)^2 = \left(x_{k} - H(y_{k}) - \alpha_{k}F(x_{k},y_{k})\right)^2\notag\\ 
&\quad = \|\xhat_{k}\|^2 -2\alpha_{k}(x_{k} - H(y_{k}))^TF(x_{k},y_{k}) + \alpha_{k}^2\|F(x_{k},y_{k})\|^2\notag\\ 
&\quad = \|\xhat_{k}\|^2 -2\alpha_{k}(x_{k} - H(y_{k}))^T(F(x_{k},y_{k}) - F(H(y_{k}),y_{k}))  + \alpha_{k}^2\|F(x_{k},y_{k})- F(H(y_{k}),y_{k})\|^2\notag\\
&\quad \stackrel{\eqref{assump:sm:F:ineq}}{\leq} \|\xhat_{k}\|^2 -2\mu_{F}\alpha_{k}\|x_{k} - H(y_{k})\|^2 + L_{F}^2\alpha_{k}^2 \|x_{k} - H(y_{k})\|^2\notag\\
&\quad \stackrel{\eqref{lem:FG_bounded:B}}{\leq} (1-2\mu_{F}\alpha_{k}+B^2\alpha_{k}^2)\|\xhat_{k}\|^2,\label{lem:xhat:eq1a}
\end{align}
where in the first inequality we also use \eqref{assump:smooth:FH:ineqF}. Second, by  \eqref{assump:smooth:FH:ineqH}, \eqref{assump:G:smooth}, \eqref{lem:FG_bounded:B}, and $G(H(y^{\star}),y^{\star}) = 0$ we consider
\begin{align}
\|G(x_{k},y_{k})\|^2 &\leq \big(\|G(x_{k},y_{k}) - G(H(y_{k}),y_{k})\| +  \|G(H(y_{k}),y_{k}) - G(H(y^{\star}),y^{\star})\|\big)^2\notag\\
&\leq \left(B\|\xhat_{k}\| + B\left(\|H(y_{k}) - H(y^{\star})\|+ \|\yhat_{k}\|\right)\right)^2\leq \left(B\|\xhat_{k}\| + B(B+1) \|\yhat_{k}\|\right)^2\notag\\
&\leq 2B^2\|\xhat_{k}\|^2 + 2B^2(B+1)^2\|\yhat_{k}\|^2.\label{lem:xhat:eq1b1}
\end{align}
Using \eqref{assump:smooth:FH:ineqH}, \eqref{alg:xy}, \eqref{lem:FG_bounded:B}, and \eqref{lem:xhat:eq1b1} we consider the second term on the right-hand side of \eqref{lem:xhat:eq1}
\begin{align}
&\|H(y_{k})-H(y_{k+1}) - \alpha_{k}\psi_{k}\|^2 \leq 2\|H(y_{k})-H(y_{k+1})\|^2  + 2\alpha_{k}^2\| \psi_{k}\|^2\notag\\
&\quad \stackrel{\eqref{lem:FG_bounded:B}}{\leq} 2B^2\|y_{k+1}-y_{k}\|^2 + 2\alpha_{k}^2\| \psi_{k}\|^2 = 2B^2\beta_{k}^2\|G(x_{k},y_{k}) + \zeta_{k}\|^2 + 2\alpha_{k}^2\|\psi_{k}\|^2\notag\\ 
&\quad  \leq  4B^2\beta_{k}^2\|G(x_{k},y_{k})\|^2 + 4B^2\beta_{k}^2\|\zeta_{k}\|^2 + 2\alpha_{k}^2\|\psi_{k}\|^2\notag\\
&\quad \stackrel{\eqref{lem:xhat:eq1b1}}{\leq} 8B^{4}\beta_{k}^2\|\xhat_{k}\|^2 + 8B^{4}(B+1)^2\beta_{k}^2\|\yhat_{k}\|^2  + 4B^2\beta_{k}^2\|\zeta_{k}\|^2 + 2\alpha_{k}^2\|\psi_{k}\|^2.\label{lem:xhat:eq1b}
\end{align}
Third, applying the same line of analysis as in \eqref{lem:xhat:eq1b1} we obtain 
\begin{align*}
\|G(x_{k},y_{k})\| &\leq \big(\|G(x_{k},y_{k}) - G(H(y_{k}),y_{k})\| +  \|G(H(y_{k}),y_{k}) - G(H(y^{\star}),y^{\star})\|\big)\notag\\
&\leq \left(B\|\xhat_{k}\| + B\left(\|H(y_{k}) - H(y^{\star})\|+ \|\yhat_{k}\|\right)\right)\leq B\|\xhat_{k}\| + B(B+1) \|\yhat_{k}\|.
\end{align*}
Using the preceding relation, \eqref{assump:smooth:FH:ineqH}, and $F(H(y_{k}),y_{k}) = 0$ we consider
\begin{align*}
&(\xhat_{k} - \alpha_{k}F(x_{k},y_{k}))^T(H(y_{k})-H(y_{k+1})) \leq (\|\xhat_{k}\| + \alpha_{k}\|F(x_{k},y_{k})\|)\,\|H(y_{k})-H(y_{k+1})\|\\
&\quad \leq B\beta_{k}(\|\xhat_{k}\| + \alpha_{k}\|F(x_{k},y_{k})\|)\|y_{k+1} - y_{k}\|\notag\\
&\quad = B\beta_{k}(\|\xhat_{k}\| + \alpha_{k}\|F(x_{k},y_{k})-F(H(y_{k}),y_{k})\|)\|G(x_{k},y_{k})+\zeta_{k}\|\notag\\ 
&\quad \leq B\beta_{k}(\|\xhat_{k}\| + B\alpha_{k}\|\xhat_{k}\|)(\|G(x_{k},y_{k})\|+\|\zeta_{k}\|)\notag\\ 
&\quad \leq B(1+B\alpha_{k})\beta_{k}\|\xhat_{k}\|(B\|\xhat_{k}\| + B(B+1) \|\yhat_{k}\| +\|\zeta_{k}\|),
\end{align*}
which by using the relation $2ab \leq a^2/\eta + \eta b^2$ for any $\eta > 0$ yields
\begin{align}
&(\xhat_{k} - \alpha_{k}F(x_{k},y_{k}))^T(H(y_{k})-H(y_{k+1}))\notag\\ 
&\quad \leq \frac{\mu_{F}}{2}\alpha_{k}\|\xhat_{k}\|^2 +  \frac{2B^2(1+B\alpha_{k})^2\beta_{k}^2}{\mu_{F}\alpha_{k}}(B\|\xhat_{k}\| + B(B+1) \|\yhat_{k}\| + \|\zeta_{k}\|)^2\notag\\ 
&\quad \leq \frac{\mu_{F}}{2}\alpha_{k}\|\xhat_{k}\|^2 +  \frac{4B^2(1+B\alpha_{k})^2\beta_{k}^2}{\mu_{F}\alpha_{k}}(B\|\xhat_{k}\| + B(B+1) \|\yhat_{k}\|)^2 + \frac{4B^2(1+B\alpha_{k})^2\beta_{k}^2}{\mu_{F}\alpha_{k}}\|\zeta_{k}\|^2\notag\\ 
&\quad \leq \frac{\mu_{F}}{2}\alpha_{k}\|\xhat_{k}\|^2 +  \frac{16B^4(B+1)^2\beta_{k}^2}{\mu_{F}\alpha_{k}} \|\zhat_{k}\|^2 + \frac{16B^2\beta_{k}^2}{\mu_{F}\alpha_{k}}\|\zeta_{k}\|^2,\label{lem:xhat:eq1c}
\end{align}
where in the last inequality we use \eqref{notation:K*} to have $1+B\alpha_{k} \leq 2$. Finally, using \eqref{lem:FG_bounded:B} and $F(H(y_{k}),y_{k}) = 0$ we obtain 
\begin{align}
2\alpha_{k}^2\psi_{k}^TF(x_{k},y_{k}) &\leq \alpha_{k}^2\|\psi_{k}\|^2 + \alpha_{k}^2\|F(x_{k},y_{k})\|^2 = \alpha_{k}^2\|\psi_{k}\|^2 + \alpha_{k}^2\|F(x_{k},y_{k}) - F(H(y_{k}),y_{k})\|^2\notag\\
&\leq \alpha_{k}^2\|\psi_{k}\|^2 + B^2\alpha_{k}^2\|\xhat_{k}\|^2.
% &\quad\leq 8B^2(1+B)^2\alpha_{k}^2\|\zhat_{k}\|^2 + 8B^2(1+B)^2(\|y^{\star}\| + \|H(0)\| + 1)^2\alpha_{k}^2 +  B^2\alpha_{k}^2\|\xhat_{k}\|^2.
\label{lem:xhat:eq1d}
\end{align}
Taking the expectation on \eqref{lem:xhat:eq1} and using the preceding relations \eqref{lem:xhat:eq1a}--\eqref{lem:xhat:eq1d} yields
\begin{align}
\Eset[\|\xhat_{k+1}\|^2] &\leq  (1-2\mu_{F}\alpha_{k}+B^2\alpha_{k}^2)\Eset[\|\xhat_{k}\|^2]  - 2\alpha_{k}\Eset[\psi_{k}^{T}\xhat_{k}] + \alpha_{k}^2\Eset[\|\psi_{k}\|^2] + B^2\alpha_{k}^2\Eset[\|\xhat_{k}\|^2] \notag\\
&\quad + 8B^{4}\beta_{k}^2\Eset[\|\xhat_{k}\|^2] + 8B^{4}(B+1)^2\beta_{k}^2\Eset[\|\yhat_{k}\|^2]  + 4B^2\beta_{k}^2\Eset[\|\zeta_{k}\|^2] + 2\alpha_{k}^2\Eset[\|\psi_{k}\|^2]\notag\\
&\quad + \frac{\mu_{F}}{2}\alpha_{k}\Eset[\|\xhat_{k}\|^2] +  \frac{16B^4(B+1)^2\beta_{k}^2}{\mu_{F}\alpha_{k}}\Eset[ \|\zhat_{k}\|^2] + \frac{16B^2\beta_{k}^2}{\mu_{F}\alpha_{k}}\Eset[\|\zeta_{k}\|^2]\notag\\
&\leq (1-\mu_{F}\alpha_{k})\Eset[\|\xhat_{k}\|^2]  - 2\alpha_{k}\Eset[\psi_{k}^{T}\xhat_{k}]  + 3\alpha_{k}^2\Eset[\|\psi_{k}\|^2] + \Big(\frac{16B^2\beta_{k}^2}{\mu_{F}\alpha_{k}}  + 4B^2\beta_{k}^2\Big)\Eset[\|\zeta_{k}\|^2]\notag\\  
&\quad + B^2\Big(\alpha_{k}^2 + 8B^2\beta_{k}^2\Big)\Eset[\|\xhat_{k}\|^2] + 8B^{4}(B+1)^2\beta_{k}^2\Eset[\|\yhat_{k}\|^2]\notag\\  
&\quad +  \frac{16B^4(B+1)^2\beta_{k}^2}{\mu_{F}\alpha_{k}}\Eset[ \|\zhat_{k}\|^2].\label{lem:xhat:eq2}
\end{align}
By \eqref{lem:psi_zeta_bound:ineq} we have
\begin{align*}
\|\psi_{k}\|^2 &\leq 8B^2(1+B)^2\|\zhat_{k}\|^2 + 8B^2(1+B)^2(\|y^{\star}\| + \|H(0)\| + 1)^2,\\
\|\zeta_{k}\|^2 &\leq 8B^2(1+B)^2\|\zhat_{k}\|^2 + 8B^2(1+B)^2(\|y^{\star}\| + \|H(0)\| + 1)^2,   
\end{align*}
which by using \eqref{lem:xhat_bias:ineq} and $\alpha_{k}\leq \alpha_{k;\tau(\alpha_{k})}$ gives
\begin{align*}
&- 2\alpha_{k}\Eset[\psi_{k}^{T}\xhat_{k}]   + 3\alpha_{k}^2\Eset[\|\psi_{k}\|^2] + \Big(\frac{16B^2\beta_{k}^2}{\mu_{F}\alpha_{k}}  + 4B^2\beta_{k}^2\Big)\Eset[\|\zeta_{k}\|^2]\notag\\
&\quad\leq 18(1+B)^{4}\alpha_{k}^2\|\zhat_{k}\|^2 + 360B^2(1+B)^{3}\alpha_{k;\tau(\alpha_{k})}\alpha_{k}\|\zhat_{k}\|^2 \notag\\
&\quad\quad + 312B(1+B)^{4}(\|y^{\star}\| + \|H(0)\| + 1)^2\alpha_{k;\tau(\alpha_{k})}\alpha_{k}\notag\\
&\quad\quad  + 8B^2(1+B)^2\Big(\frac{16B^2\beta_{k}^2}{\mu_{F}\alpha_{k}}  + 4B^2\beta_{k}^2 + 3\alpha_{k}^2\Big)\|\zhat_{k}\|^2\notag\\ 
&\quad\quad + 8B^2(1+B)^2(\|y^{\star}\| + \|H(0)\| + 1)^2\Big(\frac{4B^2(1+B\alpha_{k})^2\beta_{k}^2}{\mu_{F}\alpha_{k}}  + 4B^2\beta_{k}^2 + 3\alpha_{k}^2\Big)\notag\\
&\quad \leq 8(1+B)^4\Big(\frac{16B^2\beta_{k}^2}{\mu_{F}\alpha_{k}}  + 4B^2\beta_{k}^2 + (45B+ 6)\alpha_{k;\tau(\alpha_{k})}\alpha_{k}\Big)\|\zhat_{k}\|^2\notag\\ 
&\quad\quad + 8(1+B)^{4}(\|y^{\star}\| + \|H(0)\| + 1)^2\Big(\frac{16B^2\beta_{k}^2}{\mu_{F}\alpha_{k}}  + 4B^2\beta_{k}^2 + (39B + 3)\alpha_{k}\alpha_{k;\tau(\alpha_{k})}\Big).
\end{align*}
Substituting the preceding relation into \eqref{lem:xhat:eq2} gives
\begin{align*}
\Eset[\|\xhat_{k+1}\|^2] &\leq (1-\mu_{F}\alpha_{k})\Eset[\|\xhat_{k}\|^2]\notag\\
&\quad + 8(1+B)^4\Big(\frac{16B^2\beta_{k}^2}{\mu_{F}\alpha_{k}}  + 4B^2\beta_{k}^2 + (45B+ 6)\alpha_{k;\tau(\alpha_{k})}\alpha_{k}\Big)\|\zhat_{k}\|^2\notag\\ 
&\quad + 8(1+B)^{4}(\|y^{\star}\| + \|H(0)\| + 1)^2\Big(\frac{16B^2\beta_{k}^2}{\mu_{F}\alpha_{k}}  + 4B^2\beta_{k}^2 + (39B + 3)\alpha_{k}\alpha_{k;\tau(\alpha_{k})}\Big)\notag\\  
&\quad + B^2\Big(\alpha_{k}^2 + 8B^2\beta_{k}^2\Big)\Eset[\|\xhat_{k}\|^2] + 8B^{4}(B+1)^2\beta_{k}^2\Eset[\|\yhat_{k}\|^2] \notag\\
&\quad   +  \frac{16B^4(B+1)^2\beta_{k}^2}{\mu_{F}\alpha_{k}}\Eset[ \|\zhat_{k}\|^2]\notag\\
&\leq (1-\mu_{F}\alpha_{k})\Eset[\|\xhat_{k}\|^2]\notag\\
&\quad + 8(1+B)^4\Big(\frac{18B^2\beta_{k}^2}{\mu_{F}\alpha_{k}}  + 4B^2\beta_{k}^2 + (45B+ 6)\alpha_{k;\tau(\alpha_{k})}\alpha_{k}\Big)\|\zhat_{k}\|^2\notag\\ 
&\quad + 8(1+B)^{4}(\|y^{\star}\| + \|H(0)\| + 1)^2\Big(\frac{16B^2\beta_{k}^2}{\mu_{F}\alpha_{k}}  + 4B^2\beta_{k}^2 + (39B + 3)\alpha_{k}\alpha_{k;\tau(\alpha_{k})}\Big)\notag\\  
&\quad + B^2\alpha_{k}^2\Eset[\|\zhat_{k}\|^2] + 8B^{4}(B+1)^2\beta_{k}^2\Eset[\|\zhat_{k}\|^2]\notag\\
&\leq (1-\mu_{F}\alpha_{k})\Eset[\|\xhat_{k}\|^2] + 32(1+B)^{6}\Big(\frac{5\beta_{k}^2}{\mu_{F}\alpha_{k}}  + \beta_{k}^2 + \alpha_{k;\tau(\alpha_{k})}\alpha_{k}\Big)\|\zhat_{k}\|^2\notag\\ 
&\quad + 32(1+B)^{6}(\|y^{\star}\| + \|H(0)\| + 1)^2\Big(\frac{5\beta_{k}^2}{\mu_{F}\alpha_{k}}  + \beta_{k}^2 + \alpha_{k}\alpha_{k;\tau(\alpha_{k})}\Big),
\end{align*}
which concludes our proof. 
\end{proof}
%%%%%%%%%%
%%%%%%%%%
% \begin{lem}\label{lem:yhat}
% Suppose that Assumptions \ref{assump:smooth:FH}--\ref{assump:mixing} hold. Then we have
% \begin{align}
% \Eset[\|\yhat_{k+1}\|^2]  &\leq (1- \mu_{G}\beta_{k})\Eset[\|\yhat_{k}\|^2]   + 18(1+B)^{4}(\alpha_{k}\beta_{k} + 10B\alpha_{k;\tau(\alpha_{k})}\beta_{k} + 3\beta_{k}^2)\Eset[\|\zhat_{k}\|^2] \notag\\
% &\quad + 24(1+B)^{4}(\|y^{\star}\| + \|H(0)\| + 1)^2(\beta_{k}^2 + 7B\alpha_{k;\tau(\alpha_{k})}\beta_{k}) + \frac{B^2}{\mu_{G}}\beta_{k}\Eset[\|\xhat_{k}\|^2].\label{lem:yhat:ineq}
% \end{align}
% \end{lem}
\subsection{Proof of Lemma \ref{lem:yhat}}
\begin{proof}
Using \eqref{alg:xy} we consider
\begin{align*}
\yhat_{k+1} &= y_{k+1} - y^{\star} = \yhat_{k} - \beta_{k}G(x_{k},y_{k})  - \beta_{k}\zeta_{k} \notag\\
&= \yhat_{k} - \beta_{k}G(H(y_{k}),y_{k})- \beta_{k}\zeta_{k}  + \beta_{k} \left(G(H(y_{k}),y_{k}) - G(x_{k},y_{k})\right)  ,
\end{align*}
which implies that
\begin{align}
\|\yhat_{k+1}\|^2 &= \left\|\yhat_{k}- \beta_{k}G(H(y_{k}),y_{k})\right\|^2 - 2\beta_{k}^2\|G(H(y_{k}),y_{k})\|^2 -2\beta_{k}\yhat_{k}^T\zeta_{k} + 2\beta_{k}^2 G(H(y_{k}),y_{k})^T\zeta_{k} \notag\\
&\quad + 2\beta_{k}\yhat_{k}^T\left(G(H(y_{k}),y_{k}) - G(x_{k},y_{k})\right) +2 \beta_{k}^2G(H(y_{k}),y_{k})^TG(x_{k},y_{k})\notag\\
&\quad + \left\|\beta_{k} \left(G(H(y_{k}),y_{k}) - G(x_{k},y_{k})\right)  - \beta_{k}\zeta_{k}\right\|^2. \label{lem:yhat:Eq1}
\end{align}
We next analyze each term on the right-hand side of \eqref{lem:yhat:Eq1}. First, using ${G(H(y^{\star}),y^{\star}) = 0}$, \eqref{assump:G:sm}, \eqref{assump:smooth:FH:ineqH}, and \eqref{assump:G:smooth}  we consider 
\begin{align}
&\left\|\yhat - \beta_{k}G(H(y_{k}),y_{k})\right\|^2- 2\beta_{k}^2\|G(H(y_{k}),y_{k})\|^2  \leq \left\|\yhat_{k}\right\|^2 - 2\beta_{k}\yhat_{k}^TG(H(y_{k}),y_{k})\notag\\ 
&\quad\stackrel{\eqref{assump:G:sm}}{\leq} \|\yhat_{k}\|^2 - 2\mu_{G}\beta_{k}\|\yhat_{k}\|^2 = (1- 2\mu_{G}\beta_{k})\|\yhat_{k}\|^2.\label{lem:yhat:Eq1a}
\end{align}
Second, using \eqref{assump:G:smooth},  \eqref{lem:FG_bounded:ineq}, and $G(H(y^{\star}),y^{\star}) = 0$ we consider
\begin{align}
2\beta_{k}^2G(H(y_{k}),y_{k})^T\zeta_{k} &\leq \beta_{k}^2\|G(H(y_{k}),y_{k})\|^2 + \beta_{k}^2\|\zeta_{k}\|^2\notag\\ 
&= \beta_{k}^2\|G(H(y_{k}),y_{k}) - G(H(y^{\star}),y^{\star})\|^2 + \beta_{k}^2\|\zeta_{k}\|^2\notag\\ 
& \leq B^2\beta_{k}^2(\|H(y_{k})-H(y^{\star})\| + \|y_{k}-y^{\star}\|)^2 + \beta_{k}^2 \|\zeta_{k}\|^2\notag\\ 
&\leq B^2(B+1)^2\beta_{k}^2\|\yhat_{k}\|^2 + \beta_{k}^2\|\zeta_{k}\|^2.
\label{lem:yhat:Eq1b}
\end{align}
Third, using \eqref{assump:G:smooth}, \eqref{lem:FG_bounded:B}, and the relation $2ab\leq a^2/\eta + \eta b^2$ for all $\eta > 0$, we obtain
\begin{align}
2\beta_{k}\yhat_{k}^T\left(G(H(y_{k}),y_{k}) - G(x_{k},y_{k})\right) \leq 2B\beta_{k}\|\yhat_{k}\|\|\xhat_{k}\| \leq \mu_{G}\beta_{k}\|\yhat_{k}\|^2 + \frac{B^2}{\mu_{G}}\beta_{k}\|\xhat_{k}\|^2, \label{lem:yhat:Eq1c}
\end{align}
Next, using \eqref{assump:G:smooth}, \eqref{lem:FG_bounded:ineq}, and $G(H(y^{\star}),y^{\star}) = 0$ we have
\begin{align}
&2\beta_{k}^2 G(H(y_{k}),y_{k})^TG(x_{k},y_{k}) = 2\beta_{k}^2 \Big(G(H(y_{k}),y_{k}) - G(H(y^{\star}),y^{\star}) \Big)^T\Big(G(x_{k},y_{k}) - G(H(y^{\star}),y^{\star}) \Big)\notag\\ 
&\quad \leq 2B^2\beta_{k}^2\Big(\|H(y_{k})-H(y^{\star})\| + \|y_{k}-y^{\star}\|\Big)\|x_{k} - H(y^{\star})\| \notag\\
&\quad \leq 2B^2(B+1)\beta_{k}^2\|\yhat_{k}\|\Big(\|x_{k} - H(y_{k})\| + \|H(y_{k}) - H(y^{\star})\|\Big)\notag\\
&\quad \leq 2B^3(B+1)\beta_{k}^2\|\yhat_{k}\|^2 + 2B^2(B+1)\beta_{k}^2\|\yhat_{k}\|\|\xhat_{k}\|\notag\\
&\quad \leq 3B^3(B+1)\beta_{k}^2\|\yhat_{k}\|^2 + B(B+1)\beta_{k}^2\|\xhat_{k}\|^2\label{lem:yhat:Eq1d}
\end{align}
where the last inequality is due to the relation $2ab\leq a^2/\eta + \eta b^2$ for all $\eta > 0$. Finally, using \eqref{assump:G:smooth} and \eqref{lem:FG_bounded:B} we have
\begin{align}
\|\beta_{k} \left(G(H(y_{k}),y_{k}) - G(x_{k},y_{k})\right)  - \beta_{k}\zeta_{k}\|^2 & \leq 2\beta_{k}^2 \left\| G(H(y_{k}),y_{k}) - G(x_{k},y_{k})\right\|^2 + 2 \beta_{k}^2 \|\zeta_{k}\|^2\notag\\
&\leq 2B^2\beta_{k}^2\|\xhat_{k}\|^2 + 2 \beta_{k}^2 \|\zeta_{k}\|^2. 
\label{lem:yhat:Eq1e}
\end{align}
Taking the expectation of \eqref{lem:yhat:Eq1} and using \eqref{lem:yhat:Eq1a}--\eqref{lem:yhat:Eq1e}  yields 
\begin{align}
\Eset[\|\yhat_{k+1}\|^2] & \leq (1- 2\mu_{G}\beta_{k})\Eset[\|\yhat_{k}\|^2]   -2\beta_{k}\Eset[\yhat_{k}^T\zeta_{k}] +  B^2(B+1)^2\beta_{k}^2\Eset[\|\yhat_{k}\|^2] + \beta_{k}^2\Eset[\|\zeta_{k}\|^2]\notag\\
&\quad + \mu_{G}\beta_{k}\Eset[\|\yhat_{k}\|^2] + \frac{B^2}{\mu_{G}}\beta_{k}\Eset[\|\xhat_{k}\|^2] + 3B^3(B+1)\beta_{k}^2\Eset[\|\yhat_{k}\|^2] + B(B+1)\beta_{k}^2\Eset[\|\xhat_{k}\|^2]\notag\\
&\quad + 2B^2\beta_{k}^2\Eset[\|\xhat_{k}\|^2] + 2 \beta_{k}^2 \Eset[\|\zeta_{k}\|^2]\notag\\
&\leq (1- \mu_{G}\beta_{k})\Eset[\|\yhat_{k}\|^2]   - 2\beta_{k}\Eset[\yhat_{k}^T\zeta_{k}]  + 3 \beta_{k}^2 \Eset[\|\zeta_{k}\|^2]\notag\\
&\quad + \frac{B^2}{\mu_{G}}\beta_{k}\Eset[\|\xhat_{k}\|^2] + 3(B+1)^2\beta_{k}^2\Eset[\|\xhat_{k}\|^2]  + 4(B+1)^{4}\beta_{k}^2\Eset[\|\yhat_{k}\|^2].\label{lem:yhat:Eq2}
\end{align}
By using \eqref{lem:yhat_bias:ineq} and \eqref{lem:psi_zeta_bound:ineq} we consider
\begin{align*}
&- 2\beta_{k}\Eset[\yhat_{k}^T\zeta_{k}]  + 3 \beta_{k}^2 \Eset[\|\zeta_{k}\|^2]\notag\\
&\quad \leq 18(1+B)^{4}\alpha_{k}\beta_{k}\Eset[\|\zhat_{k}\|^2] + 180B^2(1+B)^{3}\alpha_{k;\tau(\alpha_{k})}\beta_{k}\Eset[\|\zhat_{k}\|^2] \notag\\
&\quad\quad  + 156B(1+B)^{4}(\|y^{\star}\| + \|H(0)\| + 1)^2 \alpha_{k;\tau(\alpha_{k})}\beta_{k}\notag\\
&\quad\quad +  24B^2(1+B)^2\beta_{k}^2\Eset[\|\zhat_{k}\|^2] + 24B^2(1+B)^2(\|y^{\star}\| + \|H(0)\| + 1)^2\beta_{k}^2\notag\\
&\quad\leq 18(1+B)^{4}(\alpha_{k}\beta_{k} + 10B\alpha_{k;\tau(\alpha_{k})}\beta_{k} + 2\beta_{k}^2)\Eset[\|\zhat_{k}\|^2] \notag\\
&\quad\quad + 24(1+B)^{4}(\|y^{\star}\| + \|H(0)\| + 1)^2(\beta_{k}^2 + 7B\alpha_{k;\tau(\alpha_{k})}\beta_{k}).
\end{align*}
Substituting the preceding relation into \eqref{lem:yhat:Eq2} yields
\begin{align*}
\Eset[\|\yhat_{k+1}\|^2]
&\leq (1- \mu_{G}\beta_{k})\Eset[\|\yhat_{k}\|^2]   + 18(1+B)^{4}(\alpha_{k}\beta_{k} + 10B\alpha_{k;\tau(\alpha_{k})}\beta_{k} + 2\beta_{k}^2)\Eset[\|\zhat_{k}\|^2] \notag\\
&\quad + 24(1+B)^{4}(\|y^{\star}\| + \|H(0)\| + 1)^2(\beta_{k}^2 + 7B\alpha_{k;\tau(\alpha_{k})}\beta_{k})\notag\\
&\quad + \frac{B^2}{\mu_{G}}\beta_{k}\Eset[\|\xhat_{k}\|^2] + 3(B+1)^2\beta_{k}^2\Eset[\|\xhat_{k}\|^2]  + 4(B+1)^{4}\beta_{k}^2\Eset[\|\yhat_{k}\|^2] \notag\\
&\leq (1- \mu_{G}\beta_{k})\Eset[\|\yhat_{k}\|^2]   + 18(1+B)^{4}(\alpha_{k}\beta_{k} + 10B\alpha_{k;\tau(\alpha_{k})}\beta_{k} + 3\beta_{k}^2)\Eset[\|\zhat_{k}\|^2] \notag\\
&\quad + 24(1+B)^{4}(\|y^{\star}\| + \|H(0)\| + 1)^2(\beta_{k}^2 + 7B\alpha_{k;\tau(\alpha_{k})}\beta_{k}) + \frac{B^2}{\mu_{G}}\beta_{k}\Eset[\|\xhat_{k}\|^2],
\end{align*}
which concludes our proof. 
\end{proof}
%%%%%%%%%%%%
%%%%%%
\subsection{Proof of Lemma \ref{lem:zhat:upperbound}}
% \begin{lem}\label{lem:zhat:upperbound}
% Suppose that Assumptions \ref{assump:smooth:FH}--\ref{assump:mixing} hold. Let $\{\alpha_{k},\beta_{k}\}$ be chosen as in \eqref{thm:main:stepsize}. Moreover, let $D_{1}$ and $D_{2}$ be defined as 
% \begin{align}
% \begin{aligned}
% D_{1} &=  \sum_{k=0}^{\infty} \frac{\beta_{k}^2}{\mu\alpha_{k}} +  \beta_{k}^2 + \alpha_{k}\alpha_{k;\tau(\alpha_{k})} < \infty,\\
% D_{2} &= 160(1+B)^{6}(\|y^{\star}\| + \|H(0)\| + 1)^2.
% \end{aligned}
% \label{notation:CD}
% \end{align}
% Then we obtain for all $k\geq \Kcal^*$
% \begin{align}
% \Eset[\|\zhat_{k}\|^2] &\leq D \triangleq  \Eset[\|\zhat_{0}\|^2]e^{160D_{1}(B+1)^{6}}  + D_{1}D_{2}e^{320D_{1}(B+1)^{6}}.   \label{lem:zhat:upperbound:ineq}
% \end{align}
% \end{lem}
\begin{proof}
Adding \eqref{lem:xhat:ineq} to \eqref{lem:yhat:ineq} and using $\beta_{k}\leq \alpha_{k} \leq \alpha_{k;\tau(\alpha_{k})}$ we obtain
\begin{align}
\Eset[\|\zhat_{k+1}\|^2] &\leq (1-\mu_{F}\alpha_{k})\Eset[\|\xhat_{k}\|^2] + 32(1+B)^{6}\Big(\frac{5\beta_{k}^2}{\mu_{F}\alpha_{k}}  + \beta_{k}^2 + \alpha_{k;\tau(\alpha_{k})}\alpha_{k}\Big)\|\zhat_{k}\|^2\notag\\ 
&\quad + 32(1+B)^{6}(\|y^{\star}\| + \|H(0)\| + 1)^2\Big(\frac{5\beta_{k}^2}{\mu_{F}\alpha_{k}}  + \beta_{k}^2 + \alpha_{k}\alpha_{k;\tau(\alpha_{k})}\Big)\notag\\ 
&\quad + (1- \mu_{G}\beta_{k})\Eset[\|\yhat_{k}\|^2]   + 18(1+B)^{4}(\alpha_{k}\beta_{k} + 10B\alpha_{k;\tau(\alpha_{k})}\beta_{k} + 3\beta_{k}^2)\Eset[\|\zhat_{k}\|^2] \notag\\
&\quad + 24(1+B)^{4}(\|y^{\star}\| + \|H(0)\| + 1)^2(\beta_{k}^2 + 7B\alpha_{k;\tau(\alpha_{k})}\beta_{k}) + \frac{B^2}{\mu_{G}}\beta_{k}\Eset[\|\xhat_{k}\|^2]\notag\\
&\leq \Eset[\|\zhat_{k}\|^2] + 160(1+B)^{6}\Big(\frac{\beta_{k}^2}{\mu_{F}\alpha_{k}}  + \beta_{k}^2 + \alpha_{k;\tau(\alpha_{k})}\alpha_{k}\Big)\|\zhat_{k}\|^2\notag\\
&\quad + 160(1+B)^{6}(\|y^{\star}\| + \|H(0)\| + 1)^2\Big(\frac{\beta_{k}^2}{\mu_{F}\alpha_{k}}  + \beta_{k}^2 + \alpha_{k}\alpha_{k;\tau(\alpha_{k})}\Big),  \label{lem:zhat:upperbound:Eq1}
\end{align}
where the last inequality we use \eqref{thm:main:stepsize} to have 
\begin{align*}
-\mu_{F}\alpha_{k} + \frac{B^2}{\mu_{G}}\beta_{k} \leq 0.   
\end{align*}
Let $w_{k}$ satisfy $w_{0} = 1$ and
\begin{align}
w_{k} = \prod_{t=0}^{k}\left(1+160(B+1)^{6}\Big(\frac{\beta_{t}^2}{\mu_{F}\alpha_{t}} + \beta_{t}^2 + \alpha_{t}\alpha_{t;\tau(\alpha_{t})}\Big)\right),\label{notation:wk}   
\end{align}
On the one hand, using $(1+x)\leq e^{x}$ for all $x\geq 0$ and \eqref{notation:CD} we have
\begin{align}
w_{k} &\leq e^{160(B+1)^{6}\sum_{t=0}^{k}\left(\frac{\beta_{t}^2}{\mu_{F}\alpha_{t}} + \beta_{t}^2 + \alpha_{t}\alpha_{t;\tau(\alpha_{t})}\right)}\leq e^{160D_{1}(B+1)^{6}}.  \label{lem:zhat:upperbound:wk_upperbound}
\end{align}
On the other hand, using $1+x\geq e^{-x}$ for all $x\geq 0$ and \eqref{notation:CD} we obtain
\begin{align}
w_{k} &\geq e^{160(B+1)^{6}\sum_{t=0}^{k}\left(\frac{\beta_{t}^2}{\mu_{F}\alpha_{t}} + \beta_{t}^2 + \alpha_{t}\alpha_{t;\tau(\alpha_{t})}\right)}\leq e^{-160D_{1}(B+1)^{6}}.  \label{lem:zhat:upperbound:wk_lowerbound}
\end{align}
Thus, dividing both sides of \eqref{lem:zhat:upperbound:Eq1} by $w_{k+1}$ and using \eqref{lem:zhat:upperbound:wk_lowerbound} and \eqref{notation:CD} give
\begin{align*}
\frac{\Eset\left[\|\zhat_{k+1}\|^2\right]}{w_{k+1}} &\leq \frac{\Eset\left[\|\zhat_{k}\|^2\right]}{w_{k}} + \frac{D_{2}}{w_{k+1}}\Big(\frac{\beta_{k}^2}{\mu_{F}\alpha_{k}} + \beta_{k}^2 + 2\alpha_{k}\alpha_{k;\tau(\alpha_{k})}\Big)\notag\\
&\leq \Eset[\|\zhat_{0}\|^2]+ D_{2}e^{160D_{1}(B+1)^{6}}\sum_{t=0}^{k}\Big(\frac{\beta_{t}^2}{\mu_{F}\alpha_{t}} + \beta_{t}^2 + 2\alpha_{t}\alpha_{t;\tau(\alpha_{t})}\Big)\notag\\
&\leq \Eset[\|\zhat_{0}\|^2]+ D_{1}D_{2}e^{160D_{1}(B+1)^{6}},
\end{align*}
which by using Eq.\ \eqref{lem:zhat:upperbound:wk_lowerbound} immediately gives Eq.\ \eqref{lem:zhat:upperbound:ineq}.

\end{proof}

%%%%%%%%%%%%
%%%%%%%%%%

%% file: appendix.tex
%!TEX root = nonlinear_SA.tex

\appendix

\section{Proof of Lemma \ref{lem:xhat_bias}}
\begin{proof}
We consider 
\begin{align}
-\xhat_{k}^T\psi_{k}  &= -\xhat_{k}^T(\psi_{k}-\psi_{k-\tau(\alpha_{k})}) - \xhat_{k}^T\psi_{k-\tau(\alpha_{k})}\notag\\
&= -\xhat_{k}^T(\psi_{k}-\psi_{k-\tau(\alpha_{k})}) - (\xhat_{k}-\xhat_{k-\tau(\alpha_{k})})^T\psi_{k-\tau(\alpha_{k})}  - \xhat_{k-\tau(\alpha_{k})}^T\psi_{k-\tau(\alpha_{k})}. \label{lem:xhat_bias:eq1}
\end{align}
We next analyze each term on the right-hand side of \eqref{lem:xhat_bias:eq1}. First, using \eqref{assump:mixing:tau} and \eqref{notation:psi_zeta} we have $\forall k\geq\Kcal^{\star}$
\begin{align}
&\Eset[-\xhat_{k-\tau(\alpha_{k})}^T\psi_{k-\tau(\alpha_{k})}\;|\;\Qcal_{k-\tau(\alpha_{k})}] = -\xhat_{k-\tau(\alpha_{k})}^T\Eset[\psi_{k-\tau(\alpha_{k})}\;|\;\Qcal_{k-\tau(\alpha_{k})}]\notag\\
&\quad \leq \|\xhat_{k-\tau(\alpha_{k})}\|\|\Eset[\psi_{k-\tau(\alpha_{k})}\;|\;\Qcal_{k-\tau(\alpha_{k})}]\|\stackrel{\eqref{assump:mixing:ineq}}{\leq} \alpha_{k}\|\xhat_{k-\tau(\alpha_{k})}\| \leq \alpha_{k}\|\zhat_{k-\tau(\alpha_{k})}\| \notag\\
&\quad \leq \alpha_{k}\|\zhat_{k}\| + \alpha_{k}\|\zhat_{k} - \zhat_{k-\tau(\alpha_{k})}\| \leq \alpha_{k}\|\zhat_{k}\|^2 + \alpha_{k}\|\zhat_{k} - \zhat_{k-\tau(\alpha_{k})}\|^2 + 2\alpha_{k}\notag\\ 
&\quad \leq 9(1+B)^{4}\alpha_{k}\|\zhat_{k}\|^2 +  10(1+B)^{4}\big(\|y^{\star}\| + \|H(0)\|+1\big)^2\alpha_{k} ,\label{lem:xhat_bias:eq1a}
\end{align}
where the last inequality we use \eqref{lem:zhat_bound:ineq3} and \eqref{notation:K*} (i.e., $2B\alpha_{k;\tau(\alpha_{k})} \leq \log(2) \leq 1/3)$ to have for all $k\geq \Kcal^*$
\begin{align*}
\|\zhat_{k}-\zhat_{k-\tau(\alpha_{k})}\|^2 &\leq 288B^2(1+B)^4\alpha_{k;\tau(\alpha_{k})}^2\big(\|\zhat_{k}\|^2 + (\|y^{\star}\| + \|H(0)\|+1)^2\big)\notag\\ 
&\leq 8(1+B)^4\big(\|\zhat_{k}\|^2 + (\|y^{\star}\| + \|H(0)\|+1)^2\big).
\end{align*}
Second, using \eqref{assump:smooth:FH:ineqF}, \eqref{assump:smooth:samples:ineq}, and \eqref{lem:FG_bounded:B} we consider 
\begin{align}
&-\xhat_{k}^T(\psi_{k}-\psi_{k-\tau(\alpha_{k})})\leq \|\xhat_{k}\|\|\psi_{k}-\psi_{k-\tau(\alpha_{k})})\|\leq 2B\|\xhat_{k}\|\|z_{k} -  z_{k-\tau(\alpha_{k})}\|\notag\\  
&\quad \leq 2B\|\zhat_{k}\|\|z_{k} -  z_{k-\tau(\alpha_{k})}\| \stackrel{\eqref{lem:z_bound:ineq2}}{\leq} 24B^2\alpha_{k;\tau(\alpha_{k})}\|\zhat_{k}\|(\|z_{k}\|+1)\notag\\
&\quad \stackrel{\eqref{lem:zhat_bound:eq1}}{\leq}24B^2(1+B)\alpha_{k;\tau(\alpha_{k})}\|\zhat_{k}\|^2 + 24B^2(1+B)\alpha_{k;\tau(\alpha_{k})}\|\zhat_{k}\| (\|y^{\star}\| + \|H(0)\| + 1)\notag\\
&\quad \leq 36B^2(1+B)\alpha_{k;\tau(\alpha_{k})}\|\zhat_{k}\|^2 + 12B^2(1+B)\alpha_{k;\tau(\alpha_{k})}(\|y^{\star}\| + \|H(0)\| + 1)^2,\label{lem:xhat_bias:eq1b}
\end{align}
where the last inequality we use the Cauchy-Schwarz inequality. Third, by using \eqref{lem:z_bound:ineq2} and \eqref{notation:K*} we have $\forall k\geq\Kcal^{\star}$
\begin{align*}
&\|z_{k-\tau(\alpha_{k})}\| \leq \|z_{k} - z_{k-\tau(\alpha_{k})}\| + \|z_{k}\| \stackrel{\eqref{lem:z_bound:ineq2}}{\leq} 12B\alpha_{k;\tau(\alpha_{k})}(\|z_{k}\| + 1) + \|z_{k}\|\stackrel{\eqref{notation:K*}}{\leq} 3\|z_{k}\| + 2.
\end{align*} 
Moreover, by using \eqref{lem:FG_bounded:ineq} and \eqref{notation:psi_zeta} we have
\begin{align*}
\|\psi_{k-\tau(\alpha_{k})}\| &= \|F(x_{k-\tau(\alpha_{k})},y_{k-\tau(\alpha_{k})};\xi_{k-\tau(\alpha_{k})}) - F(x_{k-\tau(\alpha_{k})}, y_{k-\tau(\alpha_{k})})\| \leq 2B\big(\|\zhat_{k-\tau(\alpha_{k})}\| + 1\big) 
\end{align*}
Using the preceding two relations, and \eqref{lem:zhat_bound:ineq2}, we have
\begin{align}
&-(\xhat_{k}-\xhat_{k-\tau(\alpha_{k})})^T\psi_{k-\tau(\alpha_{k})}\leq \|\xhat_{k}-\xhat_{k-\tau(\alpha_{k})}\|\|\psi_{k-\tau(\alpha_{k})}\|\notag\\
&\quad \stackrel{\eqref{lem:FG_bounded:ineq}}{\leq} 2B \|\xhat_{k}-\xhat_{k-\tau(\alpha_{k})}\|(\|z_{k-\tau(\alpha_{k})}\| + 1) \leq 6B \|\xhat_{k}-\xhat_{k-\tau(\alpha_{k})}\|(\|z_{k}\| + 1)\notag\\
&\quad  \stackrel{\eqref{lem:zhat_bound:eq1}}{\leq}6B(1+B) \|\xhat_{k}-\xhat_{k-\tau(\alpha_{k})}\|(\|\zhat_{k}\| + \|y^{\star}\| + \|H(0)\| + 1)\notag\\
&\quad \stackrel{\eqref{lem:zhat_bound:ineq2}}{\leq} 72B^2(1+B)^{3}\alpha_{k;\tau(\alpha_{k})}(\|\zhat_{k}\| + \|y^{\star}\| + \|H(0)\| + 1)^2\notag\\
&\quad \leq 144B^2(1+B)^{3}\alpha_{k;\tau(\alpha_{k})}\|\zhat_{k}\|^2  + 144B^2(1+B)^{3}\alpha_{k;\tau(\alpha_{k})}(\|y^{\star}\| + \|H(0)\| + 1)^2.\label{lem:xhat_bias:eq1c} 
\end{align}
Thus, taking the expectation on both sides of \eqref{lem:xhat_bias:eq1} and using \eqref{lem:xhat_bias:eq1a}--\eqref{lem:xhat_bias:eq1c} we obtain \eqref{lem:xhat_bias:ineq}, i.e., 
\begin{align*}
\Eset[-\xhat_{k}^T\psi_{k}]
&= 9(1+B)^{4}\alpha_{k}\Eset[\|\zhat_{k}\|^2] +  10(1+B)^{4}\big(\|y^{\star}\| + \|H(0)\|+1\big)^2\alpha_{k}\notag\\ 
&\quad + 36B^2(1+B)\alpha_{k;\tau(\alpha_{k})}\Eset[\|\zhat_{k}\|^2] + 12B^2(1+B)\alpha_{k;\tau(\alpha_{k})}(\|y^{\star}\| + \|H(0)\| + 1)^2\notag\\
&\quad + 144B^2(1+B)^{3}\alpha_{k;\tau(\alpha_{k})}\Eset[\|\zhat_{k}\|^2]  + 144B^2(1+B)^{3}\alpha_{k;\tau(\alpha_{k})}(\|y^{\star}\| + \|H(0)\| + 1)^2\notag\\  
&\leq 9(1+B)^{4}\alpha_{k}\Eset[\|\zhat_{k}\|^2] + 180B^2(1+B)^{3}\alpha_{k;\tau(\alpha_{k})}\Eset[\|\zhat_{k}\|^2] \notag\\
&\quad + 156B(1+B)^{4}\alpha_{k;\tau(\alpha_{k})}(\|y^{\star}\| + \|H(0)\| + 1)^2,
\end{align*}
where in the last inequality we use $\alpha_{k} \leq \alpha_{k;\tau(\alpha_{k})}$.
\end{proof}

%%%%%%%%%%
%%%%%%%%%%%
\section{Proof of Lemma \ref{lem:yhat_bias}}

\begin{proof}
We consider 
\begin{align}
-\yhat_{k}^T\zeta_{k}  &= -\yhat_{k}^T(\zeta_{k}-\zeta_{k-\tau(\alpha_{k})}) - \yhat_{k}^T\zeta_{k-\tau(\alpha_{k})}\notag\\
&= -\yhat_{k}^T(\zeta_{k}-\zeta_{k-\tau(\alpha_{k})}) - (\yhat_{k}-\yhat_{k-\tau(\alpha_{k})})^T\zeta_{k-\tau(\alpha_{k})} - \yhat_{k-\tau(\alpha_{k})}^T\zeta_{k-\tau(\alpha_{k})}. \label{lem:yhat_bias:eq1}
\end{align}
We next analyze each term on the right-hand side of \eqref{lem:xhat_bias:eq1}. First, using \eqref{assump:mixing:ineq} and \eqref{lem:zhat_bound:ineq2} we have $\forall k\geq\Kcal^{\star}$
\begin{align}
&\Eset[-\yhat_{k-\tau(\alpha_{k})}^T\zeta_{k-\tau(\alpha_{k})}\;|\;\Qcal_{k-\tau(\alpha_{k})}] = -\yhat_{k-\tau(\alpha_{k})}^T\Eset[\zeta_{k-\tau(\alpha_{k})}\;|\;\Qcal_{k-\tau(\alpha_{k})}]\notag\\
&\quad\leq \|\yhat_{k-\tau(\alpha_{k})}\|\|\Eset[\zeta_{k-\tau(\alpha_{k})}\;|\;\Qcal_{k-\tau(\alpha_{k})}]\stackrel{\eqref{assump:mixing:ineq}}{\leq} \alpha_{k}\|\yhat_{k-\tau(\alpha_{k})}\|\leq \alpha_{k}\|\zhat_{k-\tau(\alpha_{k})}\| \notag\\
&\quad \leq \alpha_{k}\|\zhat_{k}\| + \alpha_{k}\|\zhat_{k} - \zhat_{k-\tau(\alpha_{k})}\| \leq \alpha_{k}\|\zhat_{k}\|^2 + \alpha_{k}\|\zhat_{k} - \zhat_{k-\tau(\alpha_{k})}\|^2 + 2\alpha_{k}\notag\\ 
&\quad \leq 9(1+B)^{4}\alpha_{k}\|\zhat_{k}\|^2 +  10(1+B)^{4}\big(\|y^{\star}\| + \|H(0)\|+1\big)^2\alpha_{k} ,\label{lem:yhat_bias:eq1a}
\end{align}
where the last inequality we use \eqref{lem:zhat_bound:ineq3} and \eqref{notation:K*} (i.e., $2B\alpha_{k;\tau(\alpha_{k})} \leq \log(2) \leq 1/3)$ to have for all $k\geq \Kcal^*$
\begin{align*}
\|\zhat_{k}-\zhat_{k-\tau(\alpha_{k})}\|^2 &\leq 288B^2(1+B)^4\alpha_{k;\tau(\alpha_{k})}^2\big(\|\zhat_{k}\|^2 + (\|y^{\star}\| + \|H(0)\|+1)^2\big)\notag\\ 
&\leq 8(1+B)^4\big(\|\zhat_{k}\|^2 + (\|y^{\star}\| + \|H(0)\|+1)^2\big).
\end{align*}
Second, similar to \eqref{lem:xhat_bias:eq1b} we obtain 
\begin{align}
&-\yhat_{k}^T(\zeta_{k}-\zeta_{k-\tau(\alpha_{k})})\leq \|\yhat_{k}\|\|\zeta_{k}-\zeta_{k-\tau(\alpha_{k})})\|\leq 2B\|\yhat_{k}\|\|z_{k} -  z_{k-\tau(\alpha_{k})}\| \leq 2B\|\zhat_{k}\|\|z_{k} -  z_{k-\tau(\alpha_{k})}\notag\\
&\quad\leq 36B^2(1+B)\alpha_{k;\tau(\alpha_{k})}\|\zhat_{k}\|^2 + 12B^2(1+B)\alpha_{k;\tau(\alpha_{k})}(\|y^{\star}\| + \|H(0)\| + 1)^2.\label{lem:yhat_bias:eq1b}
\end{align}
Third, using the same line of analysis as in \eqref{lem:xhat_bias:eq1c} we have
\begin{align}
&-(\yhat_{k}-\yhat_{k-\tau(\alpha_{k})})^T\zeta_{k-\tau(\alpha_{k})}\leq \|\yhat_{k}-\yhat_{k-\tau(\alpha_{k})}\|\|\zeta_{k-\tau(\alpha_{k})}\|\stackrel{\eqref{lem:FG_bounded:ineq}}{\leq} 2B \|\yhat_{k}-\yhat_{k-\tau(\alpha_{k})}\|(\|z_{k-\tau(\alpha_{k})}\| + 1)\notag\\
&\quad\leq 144B^2(1+B)^{3}\alpha_{k;\tau(\alpha_{k})}\|\zhat_{k}\|^2 + 144B^2(1+B)^{3}\alpha_{k;\tau(\alpha_{k})}(\|y^{\star}\| + \|H(0)\| + 1)^2.\label{lem:yhat_bias:eq1c} 
\end{align}
Thus, taking the expectation on both sides of \eqref{lem:yhat_bias:eq1} and using \eqref{lem:yhat_bias:eq1a}--\eqref{lem:yhat_bias:eq1c} we obtain \eqref{lem:yhat_bias:ineq}. 
% \begin{align*}
% \Eset[-\yhat_{k}^T\zeta_{k}]
% &= 36(1+B)B^2\alpha_{k;\tau(\alpha_{k})}\|\zhat_{k}\|^2 + 12(1+B)B^2\alpha_{k;\tau(\alpha_{k})}(\|y^{\star}\| + \|H(0)\| + 1)^2\notag\\
% &\quad + 144B^2(1+B)^{3}\alpha_{k;\tau(\alpha_{k})}\Eset[\|\zhat_{k}\|^2] + 144B^2(1+B)^{3}\alpha_{k;\tau(\alpha_{k})}(\|y^{\star}\| + \|H(0)\| + 1)^2\notag\\ 
% &\quad + \alpha_{k}\Eset[\|\yhat_{k}\|^2] +  (1+B)^2\alpha_{k}\Eset[\|\zhat_{k}\|^2] + 2(1+B)^2\big( \|y^{\star}\| + \|H(0)\|+ 2\big)\alpha_{k}\notag\\
% &\leq \alpha_{k}\Eset[\|\yhat_{k}\|^2] + 144B(1+B)^{4}\alpha_{k;\tau(\alpha_{k})}\Eset[\|\zhat_{k}\|^2]\notag\\
% &\quad + 144B(B+1)^{4}\alpha_{k;\tau(\alpha_{k})}(\|y^{\star}\| + \|H(0)\| + 1)^2,
% \end{align*}
% where we use $\alpha_{k}\leq \alpha_{k;\tau(\alpha_{k})}$.
\end{proof}
%%%%%%%%%%%%%%%%%%%
%%%%%%%%%%%%%%%%%%%
\section{Proof of Lemma \ref{lem:z_bound}}

\begin{proof}
By \eqref{alg:xy} and since $\beta_{k}\leq \alpha_{k}$ we have for all $k\geq \Kcal^{\star}$
\begin{align*}
\|z_{k+1}\| &\leq \|z_{k}\| + \alpha_{k}\left(\|F(x_{k},y_{k};\xi_{k})\| + \frac{\beta_{k}}{\alpha_{k}}\|G(x_{k},y_{k};\xi_{k})\| \right) \stackrel{\eqref{lem:FG_bounded:ineq}}{\leq} (1+2B\alpha_{k})\|z_{k}\| + 2B\alpha_{k},
\end{align*}
which by using the relation $1+x\leq exp(x)$ for all $x\geq 0$ yields for all $k\geq \Kcal^{\star}$ and  $ t\in[k-\tau(\alpha_{k}),k]$
\begin{align*}
\|z_{t+1}\| &\leq \prod_{\ell=k-\tau(\alpha_{k})}^{t}(1+2B\alpha_{\ell})\|z_{k-\tau(\alpha_{k})}\| + 2B \sum_{\ell=k-\tau(\alpha_{k})}^{t}\alpha_{t}\prod_{u=\ell+1}^{t}(1+2B\alpha_{u})\notag\\
&\leq \|z_{k-\tau(\alpha_{k})}\|\exp\big(2B\sum_{\ell = k-\tau(\alpha_{k}}^{t}\alpha_{\ell}\big) + 2B\sum_{\ell=k-\tau(\alpha_{k})}^{t}\alpha_{t}\exp\big(2B\sum_{u = \ell+1}^{t}\alpha_{u}\big)\notag\\
&\leq 2\|z_{k-\tau(\alpha_{k})}\| + 4B\alpha_{t;\tau(\alpha_{k})},
\end{align*}
where the last inequality is due to
\begin{align*}
\sum_{t=k-\tau(\alpha_{k})}^{k}\alpha_{t} \leq \tau(\alpha_{k})\alpha_{k-\tau(\alpha_{k})} \leq \frac{\log(2)}{2B},\quad \forall k \geq\Kcal^{\star}.     
\end{align*}
Using the relation above we obtain \eqref{lem:z_bound:ineq1}, i.e.,
\begin{align*}
\|z_{k} - z_{k-\tau(\alpha_{k})}\| &\leq \sum_{t=k-\tau(\alpha_{k})}^{k-1}\|z_{t+1}-z_{t}\| \leq \sum_{t=k-\tau(\alpha_{k})}^{k-1}2B\alpha_{t}(\|z_{t}\|+1)\notag\\
&\leq \sum_{t=k-\tau(\alpha_{k})}^{k-1}2B\alpha_{t}(2\|z_{k-\tau(\alpha_{k})}\| + 4B\alpha_{t;\tau(\alpha_{k})} + 1)\notag\\
&\leq 4B\alpha_{k;\tau(\alpha_{k})}\|z_{k-\tau(\alpha_{k})}\| + 4B\alpha_{k;\tau(\alpha_{k})},
\end{align*}
where in the last inequality we use $4B\alpha_{k;\tau(\alpha_{k})} \leq 2\log(2) \leq 1$ for all $k\geq \Kcal^{\star}$. Finally, using the triangle inequality the preceding relation yields 
\begin{align*}
\|z_{k} - z_{k-\tau(\alpha_{k})}\| &\leq    4B\alpha_{k;\tau(\alpha_{k})}(\|z_{k} - z_{k-\tau(\alpha_{k})}\| + \|z_{k}\|) + 4B\alpha_{k;\tau(\alpha_{k})},
\end{align*}
which by rearranging both sides and using $4B\alpha_{k;\tau(\alpha_{k})}\leq 2\log(2)\leq 2/3$ gives \eqref{lem:z_bound:ineq2}. 
\end{proof}

%%%%%%%%%%%%%%%%%%%
%%%%%%%%%%%%%%%%%%%

%%%%%%%%%%%%%%%%%%%
%%%%%%%%%%%%%%%%%%%

\section{Proof of Lemma \ref{lem:zhat_bound}}

\begin{proof}
By \eqref{alg:xyhat} and \eqref{assump:smooth:FH:ineqH} we have
\begin{align*}
&\|\xhat_{k}\| = \|x_{k}-H(y_{k})\| \geq \|x_{k}\| - L\|y_{k}\| - \|H(0)\|,\notag\\
&\|\yhat_{k}\| = \|y_{k}-y^{\star}\| \geq \|y_{k}\| - \|y^{\star}\|,
\end{align*}
which by  \eqref{lem:FG_bounded:B} implies that 
\begin{align*}
&\|y_{k}\| \leq \|\yhat_{k}\| + \|y^{\star}\|,\notag\\
&\|x_{k}\| \leq \|\xhat_{k}\| + B\|\yhat_{k}\| + \|H(0)\|  + L\|y^{\star}\|.
\end{align*}
Thus, using $z = [x,y]^T$ and $\zhat = [\xhat,\yhat]^T$ we obtain $\forall k\geq0$
\begin{align}
\|z_{k}\| \leq (1+B)\|\zhat_{k}\| + (1+B)(\|y^{\star}\| + \|H(0)\|).    \label{lem:zhat_bound:eq1}
\end{align}
We now use \eqref{lem:FG_bounded:B}, \eqref{alg:xyhat} and \eqref{lem:z_bound:ineq1} to obtain \eqref{lem:zhat_bound:ineq1}, i.e., $\forall k\geq \Kcal^{\star}$
\begin{align}
& \|\zhat_{k}-\zhat_{k-\tau(\alpha_{k})}\| = \|\xhat_{k}-\xhat_{k-\tau(\alpha_{k})}\| + \|\yhat_{k}-\yhat_{k-\tau(\alpha_{k})}\|  \notag\\
&\quad =\|x_{k}-x_{k-\tau(\alpha_{k})} - H(y_{k}) + H(y_{k-\tau(\alpha_{k})})\|  + \|y_{k}-y_{k-\tau(\alpha_{k})}\| \notag\\
&\quad \leq \|x_{k}-x_{k-\tau(\alpha_{k})}\| + (B+1)\|y_{k} - y_{k-\tau(\alpha_{k})}\| \leq (1+B)\|z_{k}-z_{k-\tau(\alpha_{k})}\| \label{lem:zhat_bound:eq2a}
\\
&\quad \stackrel{\eqref{lem:z_bound:ineq1}}{\leq} 4B(1+B)\alpha_{k;\tau(\alpha_{k})}(\|z_{k-\tau(\alpha_{k})}\|+1)\notag\\
&\quad \stackrel{\eqref{lem:zhat_bound:eq1}}{\leq}4B(1+B)^2\alpha_{k;\tau(\alpha_{k})}\Big(\|\zhat_{k-\tau(\alpha_{k})}\| + (\|y^{\star}\| + \|H(0)\|+1)\Big).\notag 
\end{align}
Next, by \eqref{lem:zhat_bound:eq2a} and \eqref{lem:z_bound:ineq2} we achieve \eqref{lem:zhat_bound:ineq2}, i.e., $\forall k \geq \Kcal^{\star}$ 
\begin{align*}
&\|\zhat_{k}-\zhat_{k-\tau(\alpha_{k})}\| \leq 12B(1+L)\alpha_{k;\tau(\alpha_{k})}(\|z_{k}\|+1)\notag\\
&\quad \stackrel{\eqref{lem:zhat_bound:eq1}}{\leq} 12B(1+B)^2\alpha_{k;\tau(\alpha_{k})}\|\zhat_{k}\|  + 12B(1+B)^2\alpha_{k;\tau(\alpha_{k})}(\|y^{\star}\| + \|H(0)\|+1).
\end{align*}
Finally, using the relation $(a+b)^2 \leq 2a^2+2b^2$ and the preceding equation immediately gives \eqref{lem:zhat_bound:ineq3}. 
\end{proof}
%%%%%%%%%%
%%%%%%%%%%%

\subsection{Proof of Lemma \ref{lem:psi_zeta_bound}}

\begin{proof}
By \eqref{lem:FG_bounded:B}, \eqref{notation:psi_zeta},  \eqref{lem:FG_bounded:ineq}, and \eqref{lem:zhat_bound:eq1} we have
\begin{align*}
&\|\psi_{k}\| \leq 2B(\|z_{k}\|+1) \stackrel{\eqref{lem:zhat_bound:eq1} }{\leq} 2B(1+B)\|\zhat_{k}\| + 2B(1+B)(\|y^{\star}\| + \|H(0)\| + 1).
\end{align*}
The proof of $\|\zeta_{k}\|$ can be shown in a similar step. 
\end{proof}

%%%%%

%% file: nonlinear_SA.bbl
\begin{thebibliography}{10}

\bibitem{RobbinsM1951}
H.~Robbins and S.~Monro, ``A stochastic approximation method,'' {\em The Annals
  of Mathematical Statistics}, vol.~22, no.~3, pp.~400--407, 1951.

\bibitem{borkar2008}
V.~S. Borkar, {\em Stochastic Approximation: A Dynamical Systems Viewpoint}.
\newblock Cambridge University Press, 2008.

\bibitem{SBbook2018}
R.~S. Sutton and A.~G. Barto, {\em Reinforcement Learning: An Introduction}.
\newblock MIT Press, Cambridge, MA, 2nd~ed., 2018.

\bibitem{LanBook2020}
G.~Lan, {\em Lectures on Optimization Methods for Machine Learning}.
\newblock Springer-Nature, 2020.

\bibitem{WangFL2017}
M.~Wang, E.~X. Fang, and H.~Liu, ``Stochastic compositional gradient descent:
  algorithms for minimizing compositions of expected-value functions,'' {\em
  Mathematical Programming}, vol.~161, Jan 2017.

\bibitem{ZhangX2019}
J.~Zhang and L.~Xiao, ``A stochastic composite gradient method with incremental
  variance reduction,'' in {\em Advances in Neural Information Processing
  Systems 32}, pp.~9078--9088, 2019.

\bibitem{ThiemDN2020}
T.~V. Pham, T.~T. Doan, and D.~H. Nguyen, ``Distributed two-time-scale methods
  over clustered networks.'' {A}vailable at:
  \url{https://arxiv.org/abs/2010.00355}, 2020.

\bibitem{DoanBS2017}
T.~T. Doan, C.~L. Beck, and R.~Srikant, ``On the convergence rate of
  distributed gradient methods for finite-sum optimization under communication
  delays,'' {\em Proceedings ACM Meas. Anal. Comput. Syst.}, vol.~1, no.~2,
  pp.~37:1--37:27, 2017.

\bibitem{DoanMR2018b}
T.~T. Doan, S.~T. Maguluri, and J.~Romberg, ``Convergence rates of distributed
  gradient methods under random quantization: A stochastic approximation
  approach,'' {\em IEEE on Transactions on Automatic Control}, 2020.

\bibitem{BORKAR2005}
V.~S. Borkar, ``An actor-critic algorithm for constrained markov decision
  processes,'' {\em Systems \& Control Letters}, vol.~54, no.~3, pp.~207 --
  213, 2005.

\bibitem{KondaT2003}
V.~R. Konda and J.~N. Tsitsiklis, ``On actor-critic algorithms,'' {\em SIAM J.
  Control Optim.}, vol.~42, no.~4, 2003.

\bibitem{POLJAK198053}
B.~Poljak and J.~Tsypkin, ``Robust identification,'' {\em Automatica}, vol.~16,
  no.~1, pp.~53 -- 63, 1980.

\bibitem{Andrieu2003}
C.~Andrieu, N.~de~Freitas, A.~Doucet, and M.~I. Jordan, ``An introduction to
  mcmc for machine learning,'' {\em Machine Learning}, vol.~50, no.~1-2,
  pp.~5--43, 2003.

\bibitem{RamNV2009}
S.~Ram, A.~Nedić, and V.~V. Veeravalli, ``Incremental stochastic subgradient
  algorithms for convex optimization,'' {\em SIAM Journal on Optimization},
  vol.~20, no.~2, pp.~691--717, 2009.

\bibitem{JohanssonRJ2010}
B.~Johansson, M.~Rabi, and M.~Johansson, ``A randomized incremental subgradient
  method for distributed optimization in networked systems,'' {\em SIAM Journal
  on Optimization}, vol.~20, no.~3, pp.~1157--1170, 2010.

\bibitem{PolyakJ1992}
B.~T. Polyak and A.~B. Juditsky, ``Acceleration of stochastic approximation by
  averaging,'' {\em SIAM Journal on Control and Optimization}, vol.~30, no.~4,
  pp.~838--855, 1992.

\bibitem{Ruppert88}
D.~Ruppert, ``Efficient estimations from a slowly convergent robbins-monro
  process,'' {\em Technical Report 781, School of Operations Research and
  Industrial Engineering, Cornell Univ.}, 02 1988.

\bibitem{SunSY2018}
T.~Sun, Y.~Sun, and W.~Yin, ``On markov chain gradient descent,'' in {\em
  Proceedings of the 32nd International Conference on Neural Information
  Processing Systems}, NIPS’18, (Red Hook, NY, USA), p.~9918–9927, Curran
  Associates Inc., 2018.

\bibitem{Maeietal2009}
H.~R. Maei, C.~Szepesv\'{a}ri, S.~Bhatnagar, D.~Precup, D.~Silver, and R.~S.
  Sutton, ``Convergent temporal-difference learning with arbitrary smooth
  function approximation,'' in {\em Proceedings of the 22nd International
  Conference on Neural Information Processing Systems}, p.~1204–1212, 2009.

\bibitem{wu_actor_critic2020}
Y.~Wu, W.~Zhang, P.~Xu, and Q.~Gu, ``A finite time analysis of two time-scale
  actor critic methods.'' {A}vailable at:
  \url{https://arxiv.org/abs/2005.01350}, 2020.

\bibitem{Hong_actor_critic2020}
M.~Hong, H.-T. Wai, Z.~Wang, and Z.~Yang, ``A two-timescale framework for
  bilevel optimization: Complexity analysis and application to actor-critic.''
  {A}vailable at: \url{https://arxiv.org/abs/2007.05170}, 2020.

\bibitem{Khodadadian_2021}
S.~Khodadadian, T.~T. Doan, S.~T. Maguluri, and J.~Romberg, ``Finite sample
  analysis of two-time-scale natural actor-critic algorithm.'' {A}vailable at:
  \url{https://arxiv.org/abs/2101.10506}, 2021.

\bibitem{benveniste2012adaptive}
A.~Benveniste, M.~M{\'e}tivier, and P.~Priouret, {\em Adaptive algorithms and
  stochastic approximations}, vol.~22.
\newblock Springer Science \& Business Media, 2012.

\bibitem{KY2009}
H.~Kushner and G.~Yin, {\em Stochastic Approximation and Recursive Algorithms
  and Applications}.
\newblock Springer, NY, 2nd~ed., 2003.

\bibitem{Chen_SA_Envelope_2020}
Z.~Chen, S.~T. Maguluri, S.~Shakkottai, and K.~Shanmugam, ``Finite-sample
  analysis of contractive stochastic approximation using smooth convex
  envelopes,'' in {\em Proceedings of the International Conference on Neural
  Information Processing Systems}, 2020.

\bibitem{Qu_SyncSA_202}
G.~Qu and A.~Wierman, ``Finite-time analysis of asynchronous stochastic
  approximation and $q$-learning,'' in {\em Proceedings of Thirty Third
  Conference on Learning Theory}, vol.~125, pp.~3185--3205, 09--12 Jul 2020.

\bibitem{Mou_SA_2020}
W.~Mou, C.~J. Li, M.~J. Wainwright, P.~L. Bartlett, and M.~I. Jordan, ``On
  linear stochastic approximation: Fine-grained {P}olyak-{R}uppert and
  non-asymptotic concentration,'' in {\em Proceedings of Thirty Third
  Conference on Learning Theory}, vol.~125, pp.~2947--2997, 09--12 Jul 2020.

\bibitem{Karimi_colt2019}
B.~Karimi, B.~Miasojedow, E.~Moulines, and H.-T. Wai, ``Non-asymptotic analysis
  of biased stochastic approximation scheme,'' in {\em Conference on Learning
  Theory, {COLT} 2019, 25-28 June 2019, Phoenix, AZ, {USA}}, pp.~1944--1974,
  2019.

\bibitem{SrikantY2019_FiniteTD}
R.~Srikant and L.~Ying, ``Finite-time error bounds for linear stochastic
  approximation and {TD} learning,'' in {\em COLT}, 2019.

\bibitem{Chen_MC_LinearSA_2020}
S.~Chen, A.~Devraj, A.~Busic, and S.~Meyn, ``Explicit mean-square error bounds
  for monte-carlo and linear stochastic approximation,'' vol.~108 of {\em
  Proceedings of Machine Learning Research}, pp.~4173--4183, 26--28 Aug 2020.

\bibitem{ChenZDMC2019}
Z.~Chen, S.~Zhang, T.~T. Doan, S.~T. Maguluri, and J.-P. Clarke, ``{Performance
  of Q-learning with Linear Function Approximation: Stability and Finite-Time
  Analysis}.'' {A}vailable at: \url{https://arxiv.org/abs/1905.11425}, 2019.

\bibitem{BottouCN2018}
L.~Bottou, F.~Curtis, and J.~Nocedal, ``Optimization methods for large-scale
  machine learning,'' {\em SIAM Review}, vol.~60, no.~2, pp.~223--311, 2018.

\bibitem{DoanNPR2020a}
T.~T. Doan, L.~M. Nguyen, N.~H. Pham, and J.~Romberg, ``Convergence rates of
  accelerated markov gradient descent with applications in reinforcement
  learning.'' Available at: \url{https://arxiv.org/abs/2002.02873}, 2020.

\bibitem{Nagaraj_Acc_LS_NeurIPS2020}
D.~Nagaraj, X.~Wu, G.~Bresler, T.~Jain, and P.~Netrapalli, ``Least squares
  regression with markovian data: Fundamental limits and algorithms,'' in {\em
  Advances in Neural Information Processing Systems 33}, 2020.

\bibitem{KondaT2004}
V.~R. Konda and J.~N. Tsitsiklis, ``Convergence rate of linear two-time-scale
  stochastic approximation,'' {\em The Annals of Applied Probability}, vol.~14,
  no.~2, pp.~796--819, 2004.

\bibitem{DalalTSM2018}
G.~Dalal, G.~Thoppe, B.~Sz{\"o}r{\'e}nyi, and S.~Mannor, ``Finite sample
  analysis of two-timescale stochastic approximation with applications to
  reinforcement learning,'' in {\em COLT}, 2018.

\bibitem{Dalal_Szorenyi_Thoppe_2020}
G.~Dalal, B.~Szorenyi, and G.~Thoppe, ``A tale of two-timescale reinforcement
  learning with the tightest finite-time bound,'' {\em Proceedings of the AAAI
  Conference on Artificial Intelligence}, vol.~34, pp.~3701--3708, Apr. 2020.

\bibitem{DoanR2019}
T.~T. {Doan} and J.~{Romberg}, ``Linear two-time-scale stochastic approximation
  a finite-time analysis,'' in {\em 2019 57th Annual Allerton Conference on
  Communication, Control, and Computing (Allerton)}, pp.~399--406, 2019.

\bibitem{GuptaSY2019_twoscale}
H.~Gupta, R.~Srikant, and L.~Ying, ``Finite-time performance bounds and
  adaptive learning rate selection for two time-scale reinforcement learning,''
  in {\em Advances in Neural Information Processing Systems}, 2019.

\bibitem{Doan_two_time_SA2019}
T.~T. Doan, ``Finite-time analysis and restarting scheme for linear
  two-time-scale stochastic approximation.'' {A}vailable at:
  \url{https://arxiv.org/abs/1912.10583}, 2019.

\bibitem{Kaledin_two_time_SA2020}
M.~Kaledin, E.~Moulines, A.~Naumov, V.~Tadic, and H.-T. Wai, ``Finite time
  analysis of linear two-timescale stochastic approximation with {M}arkovian
  noise,'' in {\em Proceedings of Thirty Third Conference on Learning Theory},
  vol.~125, pp.~2144--2203, 2020.

\bibitem{MokkademP2006}
A.~Mokkadem and M.~Pelletier, ``Convergence rate and averaging of nonlinear
  two-time-scale stochastic approximation algorithms,'' {\em The Annals of
  Applied Probability}, vol.~16, no.~3, pp.~1671--1702, 2006.

\bibitem{Doan_two_time_SA2020}
T.~T. Doan, ``Nonlinear two-time-scale stochastic approximation: Convergence
  and finite-time performance.'' {A}vailable at:
  \url{https://arxiv.org/abs/2011.01868}, 2020.

\bibitem{Kokotovic_SP1999}
P.~Kokotovi\'{c}, H.~K. Khalil, and J.~O'Reilly, {\em Singular Perturbation
  Methods in Control: Analysis and Design}.
\newblock Society for Industrial and Applied Mathematics, 1999.

\bibitem{LevinPeresWilmer2006}
D.~A. Levin, Y.~Peres, and E.~L. Wilmer, {\em {Markov chains and mixing
  times}}.
\newblock American Mathematical Society, 2006.

\bibitem{silver2017mastering}
D.~Silver, J.~Schrittwieser, K.~Simonyan, I.~Antonoglou, A.~Huang, A.~Guez,
  T.~Hubert, L.~Baker, M.~Lai, A.~Bolton, {\em et~al.}, ``Mastering the game of
  go without human knowledge,'' {\em Nature}, vol.~550, no.~7676, pp.~354--359,
  2017.

\end{thebibliography}
